\documentclass[a4paper,12pt]{article}

\usepackage{latexsym}
\usepackage{bm}
\usepackage{amsbsy}
\usepackage[authoryear]{natbib}
\usepackage{amsmath,amsthm,amssymb,mathtools, mathrsfs}
\usepackage{titling}
\usepackage{csquotes}
\usepackage[left=2.5cm,top=2cm,right=2cm,bottom=2.5cm]{geometry} 

\newcommand{\abs}[1]{\left\lvert #1 \right\rvert}
\newcommand{\cc}{\cdot}
\newcommand{\ccs}{\hspace*{0.1cm} \cc \hspace*{0.1cm}}
\newcommand{\dd}{\hspace*{0.05cm}\mathrm{d}}
\newcommand{\given}{\hspace*{0.1cm} \middle\vert \hspace*{0.1cm}}
\newcommand{\norm}[1]{\left\lVert #1 \right\rVert}

\newcommand{\myitem}{\item[$\cc$]}

\newcommand{\R}{\mathbb{R}}

\newcommand{\A}{\mathcal{A}}

\newcommand{\C}{\mathcal{C}}
\newcommand{\D}{\mathcal{D}}
\newcommand{\E}{\mathcal{E}}
\newcommand{\F}{\mathcal{F}}

\newcommand{\K}{\mathcal{K}}
\newcommand{\M}{\mathcal{M}}
\newcommand{\N}{\mathcal{N}}
\newcommand{\V}{\mathcal{V}}

\newcommand{\U}{\mathnormal{U}}

\newcommand{\Dnull}{\mathscr{D}_0}
\newcommand{\Enull}{\mathbb{E}_0}
\newcommand{\Gnull}{\mathcal{L}_0}

\newcommand{\Pnull}{\mathbb{P}_0}

\newcommand{\cPnull}{\xrightarrow{\Pnull}}
\newcommand{\cDnull}{\xrightarrow{\Dnull}}

\newcommand{\DT}{\mathscr{D}_\theta}
\newcommand{\ET}{\mathbb{E}_\theta}
\newcommand{\GT}{\mathcal{L}_\theta}
\newcommand{\HT}{\mathcal{H}_\theta}
\newcommand{\PT}{\mathbb{P}_\theta}
\newcommand{\VT}{\mathbb{V}\mathrm{ar}_\theta}

\newcommand{\cPT}{\xrightarrow{\PT}}
\newcommand{\cDT}{\xrightarrow{\DT}}

\newcommand{\FSH}{\mathscr{H}_\theta}
\newcommand{\FSL}{\mathscr{L}^2(\mu_\theta)}

\newcommand{\vtwo}[2]{\left(\begin{array}{c} #1 \\ #2 \end{array}\right)}

\newcommand{\mtwo}[4]{\left(\begin{array}{cc} #1 & #2 \\ #3 & #4 \end{array}\right)}

\numberwithin{equation}{section}

\newtheorem{defi}{Definition}[section]

\newtheorem{cond}[defi]{Condition}
\newtheorem{lem}[defi]{Lemma}
\newtheorem{thm}[defi]{Theorem}

\newtheorem{prop}[defi]{Proposition}

\title{Prediction-based inference for integrated diffusions
    with high-frequency data}

\author{{\Large Emil S. J\o rgensen and Michael
 S\o rensen\thanks{ORCID: 0000-0001-7233-5377}} \\
Dept.\ of Mathematical Sciences, University of Copenhagen \\
Universitetsparken 5, DK-2100 Copenhagen {\O}, Denmark \\
E-mail: emil.joergensen@gmail.com and michael@math.ku.dk} \vspace{2mm} 

\date{}

\begin{document}
	\maketitle

\begin{abstract}
\vspace{1mm}
\noindent
We consider parametric inference for an ergodic and stationary diffusion process, when the data are high-frequency observations of the integral of the diffusion process. Such data are obtained via certain measurement devices, or if positions are recorded and speed is modelled by a diffusion.  In finance, realized volatility or variations thereof can be used to construct observations of the latent integrated volatility process. Specifically, we assume that the integrated process is observed at $n$ equidistant, deterministic time points $i\Delta_n$ for some $\Delta_n>0$ and consider the 
high-frequency/infinite horizon asymptotic scenario, where $n \to \infty$, $\Delta_n \to 0$ and $n\Delta_n \to \infty$. Subject to mild standard regularity conditions on $(X_t)$, we prove the asymptotic existence and uniqueness of a consistent estimator for useful and tractable classes of prediction-based estimating functions. Asymptotic normality of the estimator is obtained under the additional rate assumption $n\Delta_n^2 \to 0$. The proofs are based on the useful Euler-Itô expansions of transformations of diffusions and integrated diffusions, which we study in some detail. \\

\vspace{2mm}

\end{abstract}

	\section{Introduction}\label{sec:introduction}

Diffusion processes are used to model dynamical systems in many scientific areas, particularly in finance. While these processes are defined in terms of continuous-time dynamics, the available time series are observations of the system, or components of it,  at discrete points in time. To bridge this gap between models and data, statistical methods for discretely observed continuous-time stochastic processes is a very active area of research, where the availability of high-frequency data has generated considerable interest in the construction and study of estimators and test statistics with nice asymptotic properties as the time between consecutive observations tends to zero.

This paper deals with parametric inference for integrated diffusion models $(I_t)_{t \geq 0}$ of the general form
\begin{eqnarray}
\label{eqn:I} dI_t	&=& X_t dt, \hspace{6mm} I_0 = 0 \\
\label{eqn:X} dX_t	&=& a(X_t;\theta)dt + b(X_t;\theta)dB_t,
\end{eqnarray}
where the diffusion process $(X_t)$ takes values in an open interval $(l,r) \subseteq \R$ and is ergodic with invariant distribution $\mu_\theta$. We assume that $(X_t)$ is strictly stationary under the probability measure $\PT$, i.e.\ that $X_0 \sim \mu_\theta$.  The parameter $\theta$ takes values in $\Theta \subseteq \R^d$ for some $d \geq 1$.

Let the data be a single time series $\{I_{t^n_i}\}_{i=0}^n$ of observations of the integrated process at deterministic, equidistant points in time, i.e.\ $t^n_i=i\Delta_n$ for some $\Delta_n>0$. The process $(X_t)$ is latent. To enable consistent estimation of both drift and diffusion parameters, we consider the high-frequency/infinite horizon sampling scenario
\begin{equation}\label{eq:HF}
n \to \infty, \hspace*{0.5cm} \Delta_n \to 0, \hspace*{0.5cm} n \cc \Delta_n \to \infty,
\end{equation}
where the time horizon tends to infinity with the number of observations.
An equivalent observation scheme is given by the transformed variables
\begin{equation}\label{eq:Y}
Y_i = \Delta_n^{-1} \left(I_{t^n_i}-I_{t^n_{i-1}}\right) = \Delta_n^{-1} \int_{(i-1)\Delta_n}^{i\Delta_n} X_s \dd s, \ \ \ i=1,\ldots,n.
\end{equation}
Note that for fixed $\Delta_n$, the sequence $\{Y_i\}_{i=1}^\infty $ inherits stationary under $\PT$ from $(X_t)$.

We construct and study estimators using prediction-based estimating functions, which were proposed by \cite{pbef-2000,pbef-2011} as a versatile framework for parametric inference in non-Markovian diffusion-type models. This approach was applied to integrated diffusions in \cite{iid}. Their main contribution was to derive explicit Godambe-Heyde optimal prediction-based estimating functions for diffusions belonging to a tractable class of models that includes the Ornstein-Uhlenbeck process and the square-root (CIR) process and prove low-frequency asymptotic results. The main contribution of the present paper is to establish a high-frequency asymptotic theory for a class of prediction-based estimators, in particular, existence, uniqueness, consistency and asymptotic normality within the asymptotic scenario \eqref{eq:HF}. Our proofs build on similar results for diffusion models in \cite{ped}.

Parametric estimation for discretely observed diffusion models $(X_t)$ of the form \eqref{eqn:X} is the topic of numerous papers of which we can only list a few:
\cite{edc-1986}, \cite{edp}, \cite{mmp}, \cite{mef}, \cite{eed-1997}, \cite{Shoji1998},
\cite{imh}, \cite{mle}, \cite{eel}, \cite{bladtsorensen}, \cite{MeulenSchauer}, \cite{eed-2017}, \cite{Pilipovic24} and \cite{Eduardo}, see also the review paper \cite{sm}.

Although to a lesser extent, parametric inference for integrated diffusions has also been the topic of several papers in econometrics and statistics. In the econometric literature,  the problem appears in the guise of continuous-time stochastic volatility models. To illustrate this, consider the simple stochastic volatility model for an asset price, 
$dS_t = \sqrt{v_t} dW_t$, where $(W_t)$ denotes a standard Brownian motion. The availability of high-frequency observations of $(S_t)$ enables us to filter out discrete time observations of the latent integrated volatility, $\int_0^t v_s \dd s$, and view these as our data. Nonparametric filtering of integrated volatility from high-frequency time series is an emblematic problem in financial econometrics. An extensive list of references can be found in \cite{hffe}. This procedure has lead to the construction of estimators for integrated processes in the case where the volatility dynamics are modeled by a time-homogeneous, stationary diffusion process similar to (\ref{eqn:X}),
e.g., the GARCH(1,1) diffusion model in \cite{da}, the square-root (CIR) process in \cite{osv-1993} and the $3/2$ diffusion in \cite{ptm}. Estimation based on realized power variations that approximate the integrated volatility has been studied by e.g.\ \cite{esv}, \cite{rvsv} and \cite{svj}. \cite{mim} developed high-frequency (infill) asymptotics for GMM estimators of parameters in the diffusion coefficient of the volatility process by preliminary filtering of the spot volatility instead. Apart from the work by \cite{iid} that was summarized above, papers in the statistical literature include \cite{lid-2010}, who proposed a simulated EM-algorithm to obtain maximum likelihood estimators for integrated diffusions contaminated by noise, e.g.\ microstructure noise, and \cite{edc-2000,id-2006}, who proposed an approach that has significantly influenced the present paper. In this approach, which  is based on expansion results for small values of $\Delta_n$, the construction of contrast estimators utilizes that, as $\Delta_n \to 0$, $Y_i \approx X_{t^n_{i-1}}$, which allows high-frequency limit results for integrated diffusions to be established. Finally, nonparametric estimation of the drift and diffusion coefficient in the latent diffusion process from high-frequency observations of $(I_t)$ was studied by \cite{eid}. 

The paper is organized as follows. In Section \ref{sec:preliminaries}, we present preliminaries: the notation and concepts used in the paper, our general assumptions on $(X_t)$, and the prediction-based estimating functions considered in the paper. Section \ref{sec:expansions} contains an expansion of a transformation of the diffusion process of the form $f(X_{t^n_i}) = f(X_{t^n_{i-1}}) + \Delta_n^{1/2} \partial_x f(X_{t^n_{i-1}})b(X_{t^n_{i-1}};\theta)\varepsilon_{1,i}+\varepsilon_{2,i}$ and the similar result for $f(Y_i)$. The expansion for the integrated process, $Y_i$, was essentially pointed out by \cite{edc-2000}. These expansions serve as essential building blocks for the asymptotic theory in our paper, and because they are related to the classic Euler approximation, we refer to them as \emph{Euler-Itô expansions}. Section \ref{sec:limits} is devoted to limit theorems for integrated diffusions, while the asymptotic results on existence, uniqueness, consistency and asymptotic normality of our estimators are developed in Section \ref{sec:asymptotics}. Proofs and some auxiliary results 
are deferred to Section \ref{sec:proofs}, and Section \ref{sec:conclusion} concludes.
	\section{Preliminaries}\label{sec:preliminaries}

In this section we present the general notation used throughout the
paper and some core concepts, formulate our main assumptions on the
underlying diffusion model $(X_t)$, and define a tractable class of
prediction-based estimating functions. 

\subsection{Notation and concepts}
\label{subsec:notation}

Our general notation is as follows:

\begin{enumerate}

\item The true parameter value is denoted by $\theta_0$.
  
\item We denote the state space of $(X_t)$ by $(S,\mathscr{B}(S))$,
  where  $S=(l,r)$ for $-\infty \leq l < r \leq \infty$ is an
  open interval equipped with its Borel $\sigma$-algebra $\mathscr{B}(S)$.

\item We write $\mu_\theta(f)=\int_S f(x) \mu_\theta(dx)$ for
  functions $f:S \to \R$, and denote by $\mathscr{L}^p(\mu_\theta)$
  the space of functions $f$,  for which $\mu_\theta(|f|^p) <
  \infty$. Moreover, $\mathscr{L}_0^p(\mu_\theta)$ denotes the subset
  of $\mathscr{L}^p(\mu_\theta)$ for which  $\mu_\theta(f)=0$. 

\item By $\cPT$ and $\cDT$ we denote convergence in probability and in
  distribution under $\PT$.  

\item A function $f:S \times \Theta \to \R$ is said to be of
  \emph{polynomial growth in $x$} if there exists a $C_\theta>0$ such
  that $|f(x;\theta)| \leq C_\theta(1 + |x|^{C_\theta})$ for all $x \in S$. 

\item In this paper, $R(\Delta,x;\theta)$ denotes a generic real function such that
\begin{equation}\label{eq:R}
|R(\Delta,x;\theta)| \leq F(x;\theta),
\end{equation}
where $F$ is of polynomial growth in $x$.

\item  For real functions $f$ and $g$ defined on a measure space
  $(A,\mathscr{A}, \nu)$, we write $f \leq_C g$ if there exists a
  constant $C>0$ such that $f(a) \leq C g(a)$, for $\nu$-almost all $a
  \in A$. In particular, $f$ and $g$ can be random variables.

\item We denote by $\C^{j,k}_p(S \times \Theta)$, $j,k \geq 0$, the
  class of real-valued functions $f(x;\theta)$ satisfying that 
\begin{itemize}
\myitem $f$ is $j$ times continuously differentiable w.r.t.\ $x$;
\myitem $f$ is $k$ times continuously differentiable w.r.t.\ $\theta_1,\ldots,\theta_d$;
\myitem $f$ and all partial derivatives $\partial_x^{j_1}\partial_{\theta_1}^{k_1} \cdots \partial_{\theta_d}^{k_d}f$, $j_1 \leq j$, $k_1+\cdots+k_d \leq k$, are of polynomial growth in $x$.
\end{itemize}
We define $\C^j_p(S)$ analogously as a class of function $f : S \to \R$.

\item The \emph{infinitesimal generator} of a diffusion process $(X_t)$ is
  denoted by $\A_\theta$, and the corresponding domain by
  $\D_{\A_\theta}$. If $(X_t)$ satisfies Condition \ref{cond:X} below,
  then $\C^2_p(S) \subseteq
  \D_{\A_\theta}$, and for all $f \in \C^2_p(S)$, $\mathcal{A}_\theta f = \GT f$, where 
\begin{equation}\label{eq:GT}
\GT f(x) = a(x;\theta)\partial_xf(x) + \frac{1}{2}b^2(x;\theta)\partial_x^2f(x);
\end{equation}
see e.g. \cite{sef}.

\item For any diffusion process $(X_t)$, the \emph{potential} operator
  is given by
\begin{equation}\label{eq:U}
\U_\theta(f)(x) = \int_0^\infty P_t^\theta f(x) \dd t.
\end{equation}
It is defined for functions $f: S \to \R$ in the set $\D_{U_\theta} = \{ f :
\int_0^\infty | P_t^\theta f(x)| \dd t < \infty  \}$, where
$P_t^\theta$ denotes the \emph{transition operator}
$P_t^\theta f(x) = \ET\left(f(X_t) \given X_0=x\right)$.

\item We define
\begin{equation}\label{eq:FSH}
\FSH = \{f \in \C^4_p(S) \cap \D_{U_\theta} : \mu_\theta(f)=0, \U_\theta(f) \in \C^2_p(S)\}.
\end{equation}

\end{enumerate}

The potential operator plays an important role in our asymptotic
theory. General results ensuring that $f \in \D_{U_\theta}$ and
regularity of $\U_\theta(f)$ can be found in \cite{pe}. For an ergodic
diffusion with invariant measure $\mu_\theta$, $f \in \D_{U_\theta}$
must necessarily satisfy $\mu_\theta(f) = 0$. The reason why the
potential operator is important in our theory is that under regularity
conditions it satisfies the Poisson equation $\GT (\U_\theta(f)) =
-f$. If $(X_t)$ satisfies Condition \ref{cond:X} below, this is the
case for $f \in \FSH$, see e.g. Proposition 3.3 in \cite{ped}.

\subsection{Model assumption}
\label{subsec:model}

To establish asymptotic results for integrated diffusions of the
general form \eqref{eqn:I}-\eqref{eqn:X}, we impose the following
regularity conditions on $(X_t)$.

\begin{cond}\label{cond:X}
For any $\theta \in \Theta$, the stochastic differential equation
\begin{equation*}
dX_t = a(X_t;\theta) dt + b(X_t;\theta) dB_t, \hspace{0.2cm} X_0 \sim \mu_\theta
\end{equation*}
has a \emph{weak solution} $\left(\Omega,(\F_t),\PT,(B_t),(X_t)\right)$ for
which $\F_t = \sigma\left(X_0,(B_s)_{s \leq t}\right)$, $X_0$ is independent
of $(B_t)$ and
\begin{itemize}
\myitem $(X_t)$ is stationary and $\rho$-mixing under $\PT$.
\end{itemize}
Moreover, the triplet $(a,b,\mu_\theta)$ satisfies the regularity conditions
\begin{itemize}
\myitem $a,b \in \C^{2,0}_p(S \times \Theta)$,
\myitem $\abs{a(x;\theta)} + \abs{b(x;\theta)}\leq_C 1+|x|$,
\myitem $b(x;\theta)>0$ for $x \in S$,
\myitem $\int_S |x|^k \mu_\theta(dx) < \infty$ for all $k \geq 1$.
\end{itemize}
\end{cond}

\noindent
We define a discretized filtration by $\F^n_i:=\F_{t^n_i}$.

Easily checked conditions for $\rho$-mixing of one-dimensional
diffusion processes are given in \cite{svhm}. In particular, for
an ergodic and time-reversible diffusion process, the $\rho$-mixing
property is equivalent to the existence of a spectral gap. The latter
means that the largest non-zero eigenvalue of the generator
$\A_\theta$ of the diffusion process is strictly smaller that
zero. From spectral theory it is known that all eigenvalues are
non-positive. The size of the spectral gab, which we denote by
$\lambda_\theta$, equals minus the largest non-zero eigenvalue.

Under Condition \ref{cond:X}, it is well-known that for $f \in
\mathscr{L}_0^2(\mu_\theta)$ it holds that that $\| P_t^\theta f \|_2 \leq
e^{-\lambda t} \| f \|_2 $ for all $t \geq 0$, where $\|  f  \|_2
= \mu_\theta(f^2)^{\frac{1}{2}}$, see e.g.\ Lemma 3.2 in
\cite{ped}. Using this we can define $\int_0^\infty P_t^\theta f(x) dt$
as the $\|  \cdot \|_2$-limit of $\int_0^N P_t^\theta f (x)dt$ as $N \to
\infty$. This limit exists and belongs to $\mathscr{L}_0^2(\mu_\theta)$
because $\int_0^N P_t^\theta f dt$ is a Cauchy sequence in
$\mathscr{L}_0^2(\mu_\theta)$. Thus under Condition \ref{cond:X},
$U_\theta$ is a well-defined mapping 
$\mathscr{L}_0^2(\mu_\theta) \mapsto \mathscr{L}_0^2(\mu_\theta)$, and
since  $ \C^4_p(S) \subseteq \mathscr{L}^2(\mu_\theta)$ we have that
$\FSH \subseteq \mathscr{L}_0^2(\mu_\theta) \subseteq D_{U_\theta}$.  
In particular, the space $\FSH$ can be written as
\begin{equation}\label{eq:FSH1}
\FSH = \{f \in \C^4_p(S) : \mu_\theta(f)=0, \U_\theta(f) \in \C^2_p(S)\}.
\end{equation}

The following condition on the true parameter value $\theta_0$ is
essential to the asymptotic theory for our estimators
in Section \ref{sec:asymptotics}. Here
$\textnormal{int}(\Theta)$ denotes the interior of $\Theta$. 

\begin{cond}\label{cond:T}
The parameter space is $\Theta \subseteq \R^d$ and $\theta_0 \in
\textnormal{int}(\Theta)$. 
\end{cond}

\noindent
The notation $\mu_0=\mu_{\theta_0}$, $\Pnull=\mathbb{P}_{\theta_0}$,
etc., is applied throughout the paper. 

\subsection{Prediction-based estimating functions}
\label{subsec:predict}

Prediction-based estimating functions were proposed by
\cite{pbef-2000,pbef-2011} as a versatile framework for statistical
inference for non-Markovian diffusion-type models. In this paper, we
consider the class of estimating functions 
\begin{equation}\label{eq:pbef}
G_n(\theta) = \sum_{i=q+1}^n \sum_{j=1}^N \pi_{i-1,j}
\left[f_j(Y_i)-\breve{\pi}_{i-1,j}(\theta)\right] 
\end{equation}
where $\{f_j\}_{j=1}^N$ is a finite set of real-valued functions in
$\FSL$. For each $j \in \{1,\ldots,N\}$, $\breve{\pi}_{i-1,j}(\theta)$
denotes the orthogonal $\FSL$-projection of $f_j(Y_i)$ onto a
finite-dimensional subspace 
\begin{equation}\label{eq:pred.space}
\mathcal{P}_{i-1,j} = \text{span}\left\{1,f_j\left(Y_{i-1}\right),\ldots,f_j\left(Y_{i-q_j}\right)\right\} \subseteq \FSL,
\end{equation}
where $q_j \geq 0$. The coefficients $\pi_{i-1,j}$ in \eqref{eq:pbef}
are $d$-dimensional column vectors with entries in
$\mathcal{P}_{i-1,j}$, and $q:=\max_{1 \leq j \leq N} q_j$. \\

The subspaces $\{\mathcal{P}_{i-1,j}\}_{ij}$ are called
\emph{predictor spaces}. What we predict are values of
$f_j(Y_i)$ for $i \geq q+1$. Since every predictor space
$\mathcal{P}_{i-1,j}$ is closed, the $\FSL$-projection of $f_j(Y_i)$
onto $\mathcal{P}_{i-1,j}$, $\breve{\pi}_{i-1,j}(\theta)$, is
well-defined and uniquely determined by the normal equations 
\begin{equation}\label{eq:normal}
\ET\left(\pi\left[f_j(Y_i)-\breve{\pi}_{i-1,j}(\theta)\right]\right) = 0
\end{equation}
for all $\pi \in \mathcal{P}_{i-1,j}$. Moreover, by restricting our
attention to a stationary process $(X_t)$ and predictor spaces
of the form \eqref{eq:pred.space}, the solution to (\ref{eq:normal})
is $\breve{\pi}_{i-1,j}(\theta) = \breve{a}_n(\theta)_j^T Z_{i-1,j}$,
where 
\begin{equation*}
Z_{i-1,j}=\left(1,f_j\left(Y_{i-1}\right),\ldots,f_j\left(Y_{i-q_j}\right)\right)^T
\end{equation*}
and $\breve{a}_n(\theta)_j^T$ denotes the $(q_j+1)$-dimensional coefficient vector
\begin{equation*}
\breve{a}_n(\theta)_j^T = \left(\breve{a}_n(\theta)_{j0},\breve{a}_n(\theta)_{j1}\ldots,\breve{a}_n(\theta)_{jq_j}\right)
\end{equation*}
determined by the moment conditions
\begin{equation}\label{eq:moments}
\ET\left[Z_{q_j,j}f_j(Y_{q_j+1})\right] = \ET\left[Z_{q_j,j}Z_{q_j,j}^T\right] \breve{a}_n(\theta)_j.
\end{equation}
In the simplest case $q_j=0$, $\mathcal{P}_{i-1,j}=\text{span}\{1\}$ and, by
\eqref{eq:moments}, $\breve{\pi}_{i-1,j}(\theta) = \ET f_j(Y_1)$. \\ 

We obtain an estimator $\hat{\theta}_n$ by solving the estimating equation
$G_n(\theta)=0$,
and we call an estimator $\hat{\theta}_n$ a \emph{$G_n$-estimator} if
$\mathbb{P}_{\theta_0} (G_n(\hat{\theta}_n)=0) \to 1$ as $n \to \infty$. \\

Most prediction-based estimating functions applied in practise are of
the form considered here. In general, there is no explicit
expression for the  moments in \eqref{eq:moments}. However, as noted by
\cite{iid}, polynomial functions $f_j(y)=y^{\beta_j}$, $\beta_j \in \mathbb{N}$,
often enables calculation of the necessary moments by integrating over
mixed moments of $(X_t)$. This leads to explicit prediction-based
estimating functions for the Pearson diffusions studied in \cite{pd}.

	\section{Euler-Itô expansions}\label{sec:expansions}

This section is devoted to expansions of transformations of
diffusion processes and integrated diffusion processes observed over a
small time interval of length $\Delta_n$. We refer to these
expansions as Euler-Itô expansions. Essentially, the
following results provide a bridge between the asymptotic theory in
\cite{ped} and that of the present paper. The results are formulated
with respect to an arbitrary probability measure $\PT$. 

\subsection{Diffusion processes}

The following expansion appears in various guises in the literature on statistical inference for stochastic differential equations; see e.g.\ \cite{eed-1997}.

\begin{prop}\label{prop:XX}
Let $f \in \C^4_p(S)$. Then there exist $\F^n_i$-measurable random variables $\varepsilon_{1,i}$ and $\varepsilon_{2,i}$ such that
\begin{equation}\label{eq:XX}
f(X_{t^n_i}) = f(X_{t^n_{i-1}}) + \Delta_n^{1/2} \partial_x f(X_{t^n_{i-1}})b(X_{t^n_{i-1}};\theta)\varepsilon_{1,i}+\varepsilon_{2,i},
\end{equation}
where $\varepsilon_{1,i} \sim \N(0,1)$ and is independent of $\F^n_{i-1}$,
and $\varepsilon_{2,i}$ satisfies the moment expansions
\begin{eqnarray}
 \label{eq:exp1-XX}
\ET\left(\varepsilon_{2,i} \given \F^n_{i-1}\right) &=& \Delta_n \GT
f(X_{t^n_{i-1}}) + \Delta_n^2 R(\Delta_n,X_{t^n_{i-1}};\theta), \\
\ET\left(|\varepsilon_{2,i}|^k \given \F^n_{i-1}\right) &=& \Delta_n^k
R(\Delta_n,X_{t^n_{i-1}};\theta), \hspace{3mm} k \geq 2. \label{eq:exp2-XX}
\end{eqnarray}
\end{prop}

\subsection{Integrated diffusions}

To establish a similar result for functions of the integrated
process, we rely on earlier work by \cite{edc-2000} as well as $k$'th
order Taylor expansions of functions $f \in \C^k(S)$ of the form
\begin{equation}\label{eq:taylor}
f(Y_i) = \sum_{j=0}^{k-1} \frac{1}{j!} \partial_x^j f(X_{t^n_{i-1}}) (Y_i-X_{t^n_{i-1}})^j + \frac{1}{k!} \partial_x^k f(Z^n_i) (Y_i-X_{t^n_{i-1}})^k,
\end{equation}
where $Z^n_i$ is a random variable between $X_{t^n_{i-1}}$ and $Y_i$, i.e.\
$Z^n_i = X_{t^n_{i-1}} + s(Y_i-X_{t^n_{i-1}})$ for some $s \in (0,1)$.
The following lemma provides an upper bound for the remainder term in
\eqref{eq:taylor} for a given $k \geq 1$. 

\begin{lem}\label{lem:YX-rem}
Let $h:S \to \R$ be of polynomial growth. Then, for any $k \geq 1$,
\begin{equation}\label{eq:YX-rem}
\ET\left(|h(Z^n_i)| |(Y_i-X_{t^n_{i-1}})|^k \given \F^n_{i-1}\right) \leq_{C_k} \Delta_n^{k/2} (1+|X_{t^n_{i-1}}|)^{C_k}.
\end{equation}
\end{lem}

\noindent
In particular, if $f \in \C^1_p(S)$,
$f(Y_i) = f(X_{t^n_{i-1}}) + \partial_x f(Z^n_i) (Y_i-X_{t^n_{i-1}}),$
and Lemma \ref{lem:YX-rem} implies that
\begin{equation}\label{eq:YX-bound}
\ET\left(|f(Y_i)-f(X_{t^n_{i-1}})|^k \given \F^n_{i-1}\right) \leq_{C_k} \Delta_n^{k/2} (1+|X_{t^n_{i-1}}|)^{C_k}.
\end{equation}

\vspace*{0.2cm}

Our main result in this section is of independent interest. It is a
generalization of Proposition~2.2 in \cite{edc-2000}. Note the
resemblance with Proposition \ref{prop:XX}.  

\begin{prop}\label{prop:YX}
Let $f \in \C^4_p(S)$. Then there exist $\F^n_i$-measurable random variables $\xi_{1,i}$ and $\xi_{2,i}$ such that
\begin{equation}\label{eq:YX}
f(Y_i) = f(X_{t^n_{i-1}}) + \Delta_n^{1/2} \partial_x f(X_{t^n_{i-1}})b(X_{t^n_{i-1}};\theta)\xi_{1,i}+\xi_{2,i},
\end{equation}
where $\xi_{1,i} \sim \N(0,1/3)$ and is independent of $\F^n_{i-1}$,
and $\xi_{2,i}$ satisfies the moment expansions

\begin{eqnarray}
  \label{eq:exp1-YX}
\ET\left(\xi_{2,i} \given \F^n_{i-1}\right) &=& \Delta_n \HT
f(X_{t^n_{i-1}}) +\Delta_n^{3/2} R(\Delta_n,X_{t^n_{i-1}};\theta), \\
\ET\left(\xi^2_{2,i} \given \F^n_{i-1}\right) &=& \Delta_n^2
R(\Delta_n,X_{t^n_{i-1}};\theta), \label{eq:exp2-YX}
\end{eqnarray}
with
\begin{equation}\label{eq:HT}
\HT f(x) = \frac{1}{2} \GT f(x) - \frac{1}{12} b^2(x;\theta)\partial^2_x f(x).
\end{equation}
Moreover,
\begin{equation}\label{eq:isometry}
\ET\left(\varepsilon_{1,i} \cdot \xi_{1,i}\right) = \frac{1}{2},
\end{equation}
where $\varepsilon_{1,i}$ is the random variable that appears in
the Euler-Itô expansion \eqref{eq:XX}. 
\end{prop}

	\section{Limit theory for integrated diffusions}\label{sec:limits}

As an application of the Euler-Itô expansion \eqref{eq:YX} and the
corresponding bound \eqref{eq:YX-bound}, we derive in this section a
law of large numbers and a central limit theorem for a class of
functionals of integrated diffusions
\begin{equation}\label{eq:functional}
\frac{1}{n} \sum_{i=1}^n f(Y_i),
\end{equation}
where $f:S \to \R$ satisfies appropriate regularity conditions. For
the remainder of the paper, all asymptotic results are obtained
under the true probability measure $\Pnull$ and under the asymptotic
scenario (\ref{eq:HF}). 

\begin{lem}\label{lem:LLN}
Suppose that $f \in \C^1_p(S)$ and that $(X_t)$ satisfies Condition \ref{cond:X}. Then,
\begin{equation*}
\frac{1}{n} \sum_{i=1}^n f(Y_i) \cPnull \mu_0(f).
\end{equation*}
\end{lem}
\noindent
The result of Lemma \ref{lem:LLN} appears in a slightly stronger version in
Proposition~2 of \cite{id-2006}. \\

The result of the following lemma is that a central limit theorem for
functionals (\ref {eq:functional}) of integrated diffusions can be
obtained under the same assumption on the rate of convergence of
$\Delta_n$ and with the same Gaussian 
limit distribution as for similar functionals of discretely observed
diffusion processes; see Proposition 3.4 in \cite{ped}.

\begin{lem}\label{lem:CLT}
Assume that $f \in \mathscr{H}_0$ and that $(X_t)$ satisfies
Condition \ref{cond:X}. If $n\Delta_n^3 \to 0$, then
\begin{equation*}
\sqrt{n\Delta_n} \left(\frac{1}{n} \sum_{i=1}^n f(Y_i)\right) \cDnull \N\left(0,\V_0(f)\right),
\end{equation*}
where
\begin{equation}\label{eq:CLT.AVAR}
\V_0(f) = \mu_0\left([\partial_x U_0(f) b(\ccs;\theta_0)]^2\right) = 2\mu_0\left(f \U_0 (f)\right).
\end{equation}
\end{lem}

\vspace*{0.2cm}

The operator $U_0(f)$ appearing in the asymptotic variance
\eqref{eq:CLT.AVAR} is the potential, which was defined and discussed
in Subsections \ref{subsec:notation} and \ref{subsec:model}.

	\section{Asymptotic theory}\label{sec:asymptotics}

This section contains our main asymptotic results on $G_n$-estimators
obtained from prediction-based estimating functions of the type
described in Subsection 
\ref{subsec:predict}. The proofs are based on general asymptotic
theory for estimating functions in \cite{jjms}; see also \cite{sm}. We confine the
discussion to estimating functions of the form \eqref{eq:pbef} where
$N=1$ and simplify the notation by writing
\begin{equation}\label{eq:G}
G_n(\theta) = \sum_{i=q+1}^n \pi_{i-1} \left[f(Y_i)-\breve{\pi}_{i-1}(\theta)\right],
\end{equation}
$\mathcal{P}_{i-1}$ for the corresponding predictor spaces and so on
for objects in Subsection \ref{subsec:predict} that depend on $j$. The
extension to estimating functions with multiple predictor functions
$\{f_j\}_{j=1}^N$ is discussed in Section~4.3 in \cite{ped}. 

\subsection{Simple predictor spaces}\label{ssec:asymptotics.s}

The simplest class of estimating functions of the form \eqref{eq:G}
occurs for $q=0$. In this case, the orthogonal projection is
$\breve{\pi}_{i-1}(\theta) = \ET f(Y_1)$, and the one-dimensional
predictor space $\mathcal{P}_{i-1}$ allows us to estimate one real
parameter $\theta \in \Theta \subseteq \R$. Therefore, we consider the
one-dimensional estimating function
\begin{equation}\label{eq:PBEF.s}
G_n(\theta) = \sum_{i=1}^n \left[f(Y_i)-\ET f(Y_1)\right].
\end{equation}
Similar estimating functions were studied for discretely observed
diffusions by \cite{sef}.

Our study of the asymptotic properties of $G_n$-estimators is based on
expansions of $G_n$ in powers of $\Delta_n$. In the simple case
considered here, such an expansion follows easily from (\ref{eq:YX})
in Proposition \ref{prop:YX}, which implies that for any $f \in \C^4_p(S)$
\begin{equation}\label{eq:G.s}
\ET f(Y_1) = \mu_\theta(f) + \ET(\xi_{2,1}) = \mu_\theta(f) + \Delta_n R(\Delta_n;\theta),
\end{equation}
where $|R(\Delta_n;\theta)| \leq C(\theta)<\infty$. \\

The following regularity conditions on $G_n$ plus standard
identifiability and rate conditions ensure existence, consistency and
asymptotic normality of $G_n$-estimators. 

\begin{cond}\label{cond:G.s}
Suppose that
\begin{itemize}
\myitem $f^*(x):=f(x)-\mu_0(f) \in \mathscr{H}_0$,
\myitem $\theta \mapsto \mu_\theta(f) \in \C^1$,
\myitem For any compact subset $\M \subseteq \Theta$ and for
$\Delta_n$ sufficiently small, 
\begin{equation}
\sup_{\theta \in \M} |\partial_\theta R(\Delta_n;\theta)| \leq C(\M).
\end{equation}
\end{itemize}
\end{cond}

\vspace*{0.1cm}

\begin{thm}\label{thm:PBEF.s}
Assume Conditions \ref{cond:X}, \ref{cond:T} and \ref{cond:G.s} and
the identifiability condition $\partial_{\theta} \mu_\theta(f) \neq 0$
for all $\theta \in \Theta$. Then the following assertions hold.
\begin{itemize}
\myitem There exists a consistent sequence of $G_n$-estimators
$(\hat{\theta}_n)$ which, as $n \to \infty$, is unique in any compact
subset $\K \subseteq \Theta$ containing $\theta_0$ with
$\Pnull$-probability approaching one. 
\myitem If, moreover, $n\Delta_n^3 \to 0$, then
\begin{equation}\label{eq:AN.s}
\sqrt{n\Delta_n}\left(\hat{\theta}_n-\theta_0\right) \cDnull
\mathcal{N}\left(0,\left[\partial_{\theta}\mu_0(f)\right]^{-2}
  V_0(f)\right), 
\end{equation}
where $V_0(f) = 2 \mu_0(f^*\U_0(f^*))$.
\end{itemize}
\end{thm}

Specifically, the statement about uniqueness means that for any $G_n$-estimator
$\tilde \theta_n$ for which $\Pnull(\tilde \theta_n \in \K) \to
1$, it holds that $\Pnull(\hat{\theta}_n \neq \tilde \theta_n) \to 0$.

The identifiability condition and the assumption about the rate of
convergence of $\Delta_n$ are exactly as in the similar result for
prediction based estimating functions for discrete time
observations of diffusion processes in \cite{ped}. Also the Gaussian
the limit distribution is the same, which enables us to
use the Monte Carlo method to calculate the asymptotic variance
developed in Section~5.1 of \cite{ped}. Importantly, this method does not require
an expression for the potential.

\subsection{1-lag predictor spaces}\label{ssec:asymptotics.1}

The introduction of functions of past observations in the predictor
space $\mathcal{P}_{i-1}$ increases the mathematical complexity
considerably. Our main result establishes existence, uniqueness, consistency and
asymptotically normality for prediction-based
$G_n$-estimators with $q=1$ under appropriate regularity
conditions. In this case, the predictor space $\mathcal{P}_{i-1}$ is
spanned by $1$ and $f(Y_{i-1})$, and it follows from the normal
equations \eqref{eq:moments} that the optimal predictor is
\begin{equation*}
\breve{\pi}_{i-1}(\theta) = \breve{a}_n(\theta)_0 + \breve{a}_n(\theta)_1 f(Y_{i-1}),
\end{equation*}
where $\breve{a}_n(\theta)_0$ and $\breve{a}_n(\theta)_1$ are uniquely
determined by
\begin{eqnarray}
\label{eqn:a0} \breve{a}_n(\theta)_0 &=& \ET f(Y_1) \left(1-\breve{a}_n(\theta)_1\right), \\
\label{eqn:a1} \breve{a}_n(\theta)_1 &=& \frac{\ET\left[f(Y_1)f(Y_2)\right]-[\ET f(Y_1)]^2}{\VT f(Y_1)}.
\end{eqnarray}
Consistent with the two-dimensional predictor space, we consider $d=2$
and investigate the estimating function 
\begin{equation}\label{eq:PBEF.1}
G_n(\theta) = \sum_{i=2}^n \vtwo{1}{f(Y_{i-1})} \left[f(Y_i)-\breve{a}_n(\theta)_0-\breve{a}_n(\theta)_1f(Y_{i-1})\right]
\end{equation}
for which the expansion in powers of $\Delta_n$ is more difficult than
for \eqref{eq:PBEF.s}.

Using on the Euler-Itô expansions in Section \ref{sec:expansions}, we
start by expanding the projection coefficients $\breve{a}_n(\theta)_0$
and $\breve{a}_n(\theta)_1$. As the proof is a bit long, we formulate
the result in a separate lemma. 

\begin{lem}\label{lem:a}
For $f \in \C^4_p(S)$, the projection coefficient vector
$\breve{a}_n(\theta)=\left(\breve{a}_n(\theta)_0,\breve{a}_n(\theta)_1\right)^T$
has the expansion 
\begin{equation}\label{eq:a}
\breve{a}_n(\theta) = \vtwo{0}{1} + \Delta_n \vtwo{-K_f(\theta)\mu_\theta(f)}{K_f(\theta)} + \Delta_n^{3/2} R(\Delta_n;\theta)
\end{equation}
where $|R(\Delta_n;\theta)| \leq C(\theta)$ and
\begin{equation}\label{eq:K}
K_f(\theta) = \VT f(X_0)^{-1}\left[\mu_\theta(f \GT f) + \frac{1}{6} \mu_\theta\left([b(\ccs;\theta)\partial_x f]^2\right)\right].
\end{equation}
\end{lem}

The following regularity conditions on $G_n$ are imposed in our
asymptotic theory.

\begin{cond}\label{cond:G.1}
Suppose that
\begin{itemize}

\myitem $f^*_1(x) = K_f(\theta_0)\left[\mu_0(f)-f(x)\right] \in \mathscr{H}_0$,
\myitem $f^*_2(x) = f(x) \Gnull f(x) + \frac{1}{6}[b(x;\theta_0)\partial_x f(x)]^2 - K_f(\theta_0)f(x)\left[f(x)-\mu_0(f)\right] \in \mathscr{H}_0$,
\myitem $(\theta \mapsto \mu_\theta(f)) \in \C^1$, $(\theta \mapsto
K_f(\theta)) \in \C^1$ and the remainder term in (\ref{eq:a}) satisfies that
\begin{equation}\label{eq:dR}
\sup_{\theta \in \M} \norm{\partial_{\theta^T} R(\Delta_n;\theta)} \leq C(\M),
\end{equation}
for any compact subset $\M \subseteq \Theta$ and for $\Delta_n$ sufficiently small.
\end{itemize}
\end{cond}

\vspace*{0.2cm}

The matrix norm $\norm{\ccs}$ in \eqref{eq:dR} and \eqref{eq:W} can be chosen
arbitrarily, and for convenience we suppose that $\norm{\ccs}$ is
submultiplicative. The following lemma establishes crucial technical steps in the
proof of the main Theorem \ref{thm:PBEF.1}. 

\begin{lem}\label{lem:PBEF.1}
Assume that Conditions \ref{cond:X} and \ref{cond:G.1} holds. Then,
for any $\theta \in \Theta$, 
\begin{equation*}
(n\Delta_n)^{-1} G_n(\theta) \cPnull \gamma(\theta_0;\theta)
\end{equation*}
where
\begin{equation}\label{eq:gamma}
\gamma(\theta_0;\theta) = \vtwo{K_f(\theta) (\mu_\theta-\mu_0)(f)}{\mu_0(f \Gnull f)+\frac{1}{6}\mu_0\left([b(\ccs;\theta_0)\partial_x f]^2\right)-K_f(\theta)\left[\mu_0(f^2) - \mu_0(f)\mu_\theta(f)\right]}.
\end{equation}
Moreover, for any compact subset $\M \subseteq \Theta$
\begin{equation}\label{eq:W}
\sup_{\theta \in \M} \norm{(n\Delta_n)^{-1}\partial_{\theta^T} G_n(\theta) - W(\theta)} \cPnull 0,
\end{equation}
where 
\begin{equation*}
W(\theta) =
\mtwo{1}{\mu_0(f)}{\mu_0(f)}{\mu_0(f^2)}
\mtwo{\partial_{\theta_1}\left[K_f(\theta)\mu_\theta(f)\right]}{\partial_{\theta_2}\left[K_f(\theta)\mu_\theta(f)\right]}{-\partial_{\theta_1} K_f(\theta)}{-\partial_{\theta_2} K_f(\theta)}.
\end{equation*}
\end{lem}

\begin{thm}\label{thm:PBEF.1}
Assume Conditions \ref{cond:X}, \ref{cond:T} and \ref{cond:G.1}, 
that $W(\theta)$ is non-singular, and that the identifiability condition 
$\gamma(\theta_0;\theta) \neq 0$ for all $\theta \neq \theta_0$
is satisfied.

\vspace{2mm}

\noindent
Then the following assertions hold:
\begin{itemize}
\myitem There exists a consistent sequence of $G_n$-estimators
$(\hat{\theta}_n)$ which, as $n \to \infty$, is unique in any compact
subset $\K \subseteq \Theta$ containing $\theta_0$ with
$\Pnull$-probability approaching one. 
\myitem If, moreover, $n\Delta_n^2 \to 0$, then
\begin{equation}\label{eq:AN.1}
\sqrt{n\Delta_n}\left(\hat{\theta}_n-\theta_0\right) \cDnull \N_2\left(0,\left[W(\theta_0)^{-1} V_0(f) (W(\theta_0)^{-1})^T\right]\right),
\end{equation}
where
\begin{equation*}
V_0(f) = \mtwo
{\mu_0\left(\left[\partial_x \U_0(f^*_1)b(\ccs;\theta_0)\right]^2\right)}
{\mathbb{C}\mathrm{ov}(f)}
{\mathbb{C}\mathrm{ov}(f)}
{\mu_0\left(\left[\partial_x \U_0(f^*_2) + f \partial_x f \right]^2 b^2(\ccs;\theta_0)\right)},
\end{equation*}
with
\begin{equation*}
\mathbb{C}\mathrm{ov}(f) = \mu_0\left(\partial_x \U_0(f^*_1)\left[\partial_x U_0(f^*_2) + f \partial_x f \right] b^2(\ccs;\theta_0)\right).
\end{equation*}
\end{itemize}
\end{thm}

\vspace*{0.2cm}

Compared to the results in \cite{ped}, the lower order
$\Delta_n^{3/2}$ of the remainder term in the expansion
\eqref{eq:a} necessitates the rate assumption $n\Delta_n^2 \to 0$,
which is stronger than what is needed for discretely observed
diffusion processes. The same strong rate assumption appears in
\cite{id-2006} to ensure asymptotic normality for a class of minimum
contrast estimators with  observations of an integrated diffusion.

        \section{Proofs and auxiliary results}\label{sec:proofs}

In this section we present the proofs of the results of the paper and
some auxiliary results that are needed in the proofs. 

\subsection{Auxiliary results}

We use several times that for a diffusion process $X$ satisfying
Condition \ref{cond:X} and a function $f \in \C^1_p(S)$, there exists,
for every $k \geq 1$, a constant $C_{k,\theta}>0$ such that
\begin{equation}\label{eq:Esupk}
\ET\left(\sup_{s \in [0,\Delta]}|f(X_{t+s})-f(X_t)|^k \given
  \F_t\right) \leq C_{k,\theta} \Delta^{k/2} \left(1+|X_t|\right)^{C_{k,\theta}}.
\end{equation}
This classical result can be proved following the proofs of the similar results
in \cite{eed-1997} and \cite{edc-2000}.

We also use the well-known result that if $a(\cdot;\theta) \in \C_p^{2k,0}(S \times
\Theta)$, $b(\cdot ;\theta) \in \C_p^{2k,0}(S \times \Theta)$ and $f
\in \C^{2(k+1)}_p(S)$ for a $k \geq 0$. Then
\begin{equation}\label{eq:condmomexp}
\ET\left(f(X_{t+\Delta}) \given \F_t\right) = \sum_{i=0}^k \frac{\Delta^i}{i!} \GT^i f(X_t) +  \Delta^{k+1}R(\Delta,X_t;\theta),
\end{equation}
see e.g.\ Lemma 1.10 in \cite{sm}.

\begin{lem}\label{lemx:ibp}
Let $(X_t)_{t \geq 0}$ be a continuous semimartingale on
$(\Omega,\F,(\F_t),\mathbb{P})$,  and suppose that $(H_t)_{t \geq 0}$
is $(\F_t)$-adapted and continuous. For any $t \geq t^* \geq 0$, 
\begin{equation*}
\int_{t^*}^t \left(\int_{t^*}^s H_u dX_u\right) \dd s = \int_{t^*}^t (t-s) H_s dX_s.
\end{equation*}
\end{lem}

\begin{proof}
Without loss of generality, we can assume that $t^* = 0$. Define  $Z_t
= \int_0^t H_s dX_s$. By stochastic integration-by-parts (the Itô-
formula), $d(tZ_t) = t dZ_t + Z_t \dd t$. Thus
\begin{equation*}
\int_0^t t d Z_s = t Z_t = \int_0^t Z_s \dd s + \int_0^t s dZ_s,
\end{equation*}
which verifies the result. 
\end{proof}

\subsection{Proofs}

Since we study limits as $\Delta_n \rightarrow 0$, we can in all the
proofs assume that $\Delta_n$ is bounded from above, e.g.\ $\Delta_n
\leq 1$.

\begin{proof}[Proof of Proposition \ref{prop:XX}]
By Itô's formula,
\begin{equation*}
f(X_{t^n_i}) = f(X_{t^n_{i-1}}) + \int_{(i-1)\Delta_n}^{i\Delta_n} \GT f(X_s) \dd s + \int_{(i-1)\Delta_n}^{i\Delta_n} \partial_x f(X_s)b(X_s;\theta)dB_s.
\end{equation*}
With the definitions
\begin{eqnarray}
\label{eqn:var1}	\varepsilon_{1,i}	&=& \Delta_n^{-1/2}\int_{(i-1)\Delta_n}^{i\Delta_n} dB_s, \\
							A_i					&=& \int_{(i-1)\Delta_n}^{i\Delta_n} \GT f(X_s) \dd s, \nonumber \\
							D_i					&=& \int_{(i-1)\Delta_n}^{i\Delta_n} \left[\partial_x f(X_s)b(X_s;\theta)-\partial_x f(X_{t^n_{i-1}})b(X_{t^n_{i-1}};\theta)\right] dB_s, \nonumber \\
\label{eqn:var2} \varepsilon_{2,i}	&=& A_i + D_i, 
\end{eqnarray}
we obtain an expansion of the form
\begin{equation*}
f(X_{t^n_i}) = f(X_{t^n_{i-1}}) + \Delta_n^{1/2} \partial_x f(X_{t^n_{i-1}})b(X_{t^n_{i-1}};\theta)\varepsilon_{1,i}+\varepsilon_{2,i},
\end{equation*}
where $\varepsilon_{1,i}$ and $\varepsilon_{2,i}$ are $\F^n_i$-measurable and $\varepsilon_{1,i} \sim \N(0,1)$ and independent of $\F^n_{i-1}$. \\

To prove the conditional moment expansions \eqref{eq:exp1-XX}-\eqref{eq:exp2-XX}
apply Fubini's theorem followed by (\ref{eq:condmomexp}) to obtain
\begin{eqnarray*}
\ET\left(A_i \given \F^n_{i-1}\right)
&=& \int_0^{\Delta_n} \ET\left(\GT f(X_{t^n_{i-1}+u}) \given \F^n_{i-1}\right) \dd u \\
&=& \int_0^{\Delta_n} \left[\GT f(X_{t^n_{i-1}}) + u \cc
    R(u,X_{t^n_{i-1}};\theta)\right] \dd u \\ 
&=& \Delta_n \GT f(X_{t^n_{i-1}}) + \Delta_n^2 R(\Delta_n,X_{t^n_{i-1}};\theta).
\end{eqnarray*}
Furthermore, since $\ET\left(\int_0^t \left[\partial_x
    f(X_s)b(X_s;\theta)\right]^2 \dd s\right) =
t \mu_\theta\left([f' b(\ccs;\theta)]^2\right) <\infty$,
the stochastic integral $\int_0^t \partial_x f(X_s)b(X_s;\theta)dB_s$ is
a $\PT$-martingale, so $\ET\left(D_i \given \F^n_{i-1}\right) = 0$,
which verifies \eqref{eq:exp1-XX}. \\

For conditional moments of order $k \geq 2$, we write
\begin{equation*}
A_i = \Delta_n \GT f(X_{t^n_{i-1}}) + \int_{(i-1)\Delta_n}^{i\Delta_n} \left[\GT f(X_s)- \GT f(X_{t^n_{i-1}})\right] \dd s
\end{equation*}
and observe that, by Jensen's inequality,
\begin{eqnarray*}
\abs{\int_{(i-1)\Delta_n}^{i\Delta_n} \left[\GT f(X_s) - \GT f(X_{t^n_{i-1}})\right] \dd s}^k
& \leq & \Delta_n^{k} \cc \Delta_n^{-1} \int_{(i-1)\Delta_n}^{i\Delta_n} |\GT f(X_s) - \GT f(X_{t^n_{i-1}})|^k \dd s \\
& \leq & \Delta_n^k \sup_{u \in [0,\Delta_n]} |\GT f(X_{t^n_{i-1}+u}) - \GT f(X_{t^n_{i-1}})|^k.
\end{eqnarray*}
Hence, by (\ref{eq:Esupk}),
\begin{eqnarray*}
\lefteqn{\ET\left(|A_i|^k \given \F^n_{i-1}\right)} \\
&\leq_{C_k}&	\Delta_n^k (1+|X_{t^n_{i-1}}|)^{C_k} + \Delta_n^k \cc \ET\left(\sup_{u \in [0,\Delta_n]} |\GT f(X_{t^n_{i-1}+u}) - \GT f(X_{t^n_{i-1}})|^k \given \F^n_{i-1}\right) \\
&\leq_{C_k}&	\Delta_n^k (1+|X_{t^n_{i-1}}|)^{C_k}.
\end{eqnarray*}

Similarly, with $h(x;\theta)=\partial_x f(x)b(x;\theta)$, the
Burkholder-Davis-Gundy inequality (see e.g. \cite{dp}), Jensen's
inequality and (\ref{eq:Esupk}) imply that for all $k \geq 2$,
\begin{eqnarray*}
\lefteqn{\ET\left(\abs{D_i}^k \given \F^n_{i-1}\right)} \\
& =& \ET\left(\abs{\int_{(i-1)\Delta_n}^{i\Delta_n}
     \left[h(X_s;\theta)-h(X_{t^n_{i-1}};\theta)\right] dB_s}^k \given
     \F^n_{i-1}\right) \\ 
& \leq_{C_k}& \ET\left(\left[\int_{(i-1)\Delta_n}^{i\Delta_n}
              \left[h(X_s;\theta)-h(X_{t^n_{i-1}};\theta)\right]^2 \dd
              s\right]^{k/2} \given \F^n_{i-1}\right) \\ 
& \leq& \Delta_n^{k/2} \cc \ET\left(\Delta_n^{-1}
        \int_{(i-1)\Delta_n}^{i\Delta_n}
        |h(X_s;\theta)-h(X_{t^n_{i-1}};\theta)|^k \dd s \given
        \F^n_{i-1}\right) \\ 
& \leq& \Delta_n^{k/2} \cc \ET\left(\sup_{u
\in [0,\Delta_n]} |h(X_{t^n_{i-1}+u};\theta)-h(X_{t^n_{i-1}};\theta)|^k
\given \F^n_{i-1}\right) \\ 
& \leq_{C_k}& \Delta_n^{k} (1+|X_{t^n_{i-1}}|)^{C_k}
\end{eqnarray*}
and since $|\varepsilon_{2,i}|^k \leq_{C_k} |A_i|^k + |D_i|^k$, we
conclude that $\ET\left(|\varepsilon_{2,i}|^k \given \F^n_{i-1}\right)
= \Delta_n^k R(\Delta_n,X_{t^n_{i-1}};\theta)$.
\end{proof}

\begin{proof}[Proof of Lemma \ref{lem:YX-rem}]
Since $h$ is of polynomial growth,
\begin{equation*}
\abs{h(Z^n_i)} \leq_{C_k} 1 + |X_{t^n_{i-1}}|^{C_k} + |Y_i-X_{t^n_{i-1}}|^{C_k},
\end{equation*}
and by Jensen's inequality,
\begin{equation*}
|Y_i-X_{t^n_{i-1}}|^k \leq \Delta_n^{-1}
\int_{(i-1)\Delta_n}^{i\Delta_n} |X_s - X_{t^n_{i-1}}|^k \dd s \leq
\sup_{u \in [0,\Delta_n]} |X_{t^n_{i-1}+u}-X_{t^n_{i-1}}|^k. 
\end{equation*}
Hence the lemma follows because (\ref{eq:Esupk}) implies that for any
$k \geq 1$
\begin{equation*}
\ET\left(|Y_i-X_{t^n_{i-1}}|^k \given \F^n_{i-1} \right) \leq_{C_k} \Delta_n^{k/2}(1+|X_{t^n_{i-1}}|)^{C_k}.
\end{equation*}

\end{proof}

\begin{proof}[Proof of Proposition \ref{prop:YX}]
We start by proving the result for the identity mapping $f(x)=x$. In this
case, $f' \equiv 1$ and $f'' \equiv 0$, so the Euler-Itô expansion
\eqref{eq:YX} takes the form 
\begin{equation}\label{eq:YX-identity}
Y_i = X_{t^n_{i-1}} + \Delta_n^{1/2} b(X_{t^n_{i-1}};\theta)\xi^*_{1,i}+\xi^*_{2,i},
\end{equation}
where asterisks $(*)$ are used to distinguish the remainder
terms here from the general case. Here equation
(\ref{eq:exp1-YX}) has the form 
\begin{equation}\label{eq:xi}
\ET\left(\xi^*_{2,i} \given \F^n_{i-1}\right) = \Delta_n \frac{1}{2} a(X_{t^n_{i-1}};\theta) + \Delta_n^{3/2} R(\Delta_n,X_{t^n_{i-1}};\theta).
\end{equation}

By applying Lemma \ref{lemx:ibp} to the stochastic integral, we
find that
\begin{eqnarray*}
Y_i - X_{t^n_{i-1}}
&=& \Delta_n^{-1} \int_{(i-1)\Delta_n}^{i\Delta_n} \left(\int_{(i-1)\Delta_n}^s a(X_u;\theta) \dd u + \int_{(i-1)\Delta_n}^s b(X_u;\theta) dB_u\right) \dd s \\
&=& \Delta_n^{-1} \int_{(i-1)\Delta_n}^{i\Delta_n}
    \int_{(i-1)\Delta_n}^s a(X_u;\theta) \dd u \dd s +
    \Delta_n^{-1}\int_{(i-1)\Delta_n}^{i\Delta_n} (i\Delta_n-s)
    b(X_s;\theta) dB_s, 
\end{eqnarray*}
which, in turn, yields an expansion of the form \eqref{eq:YX-identity} by defining
\begin{eqnarray*}
\xi^*_{1,i}	&=& \Delta_n^{-3/2}\int_{(i-1)\Delta_n}^{i\Delta_n} (i\Delta_n-s) dB_s, \\
A_i &=& \Delta_n^{-1}
\int_{(i-1)\Delta_n}^{i\Delta_n} 
\int_{(i-1)\Delta_n}^s a(X_u;\theta) \dd u \dd s, \\ 
D_i			&=& \Delta_n^{-1}\int_{(i-1)\Delta_n}^{i\Delta_n} \left[b(X_s;\theta)-b(X_{t^n_{i-1}};\theta)\right](i\Delta_n-s) dB_s, \\
\xi^*_{2,i}	&=& A_i + D_i.
\end{eqnarray*}

To verify the properties of $\xi^*_{1,i}$ and $\xi^*_{2,i}$,
we observe that both are measurable
w.r.t. $\F^n_i$, $\xi^*_{1,i}$ is Gaussian and independent of
$\F^n_{i-1}$ and $\ET(\xi^*_{1,i})=0$. Moreover, by Itô's isometry 
\begin{equation*}
\ET((\xi^*_{1,i})^2) = \Delta_n^{-3}\int_{(i-1)\Delta_n}^{i\Delta_n} (i\Delta_n-s)^2 \dd s = \frac{1}{3}.
\end{equation*}
To prove the conditional moment expansions of $\xi^*_{2,i}$, we first
use the martingale property of $\int_0^t
b(X_s;\theta)(i\Delta_n-s)dB_s$ to conclude that  
\[
\ET\left(D_i \given \F^n_{i-1}\right) =	\Delta_n^{-1} \cc
\ET\left(\int_{(i-1)\Delta_n}^{i\Delta_n}
  \left[b(X_s;\theta)-b(X_{t^n_{i-1}};\theta)\right](i\Delta_n-s) dB_s
  \given \F^n_{i-1}\right)  = 0.
\]
Therefore, $\ET(\xi^*_{2,i} \hspace*{0.05cm} | \hspace*{0.05cm}
\F^n_{i-1}) = \ET(A_i  \hspace*{0.05cm} | \hspace*{0.05cm}
\F^n_{i-1})$. Application of Fubini's theorem and \eqref{eq:condmomexp} shows that 
\begin{eqnarray*}
\ET\left(A_i \given \F^n_{i-1}\right) &=& \Delta_n^{-1}\int_{(i-1)\Delta_n}^{i\Delta_n} \int_{0}^{s-t^n_{i-1}} \ET\left(a(X_{t^n_{i-1}+v};\theta) \given \F^n_{i-1}\right) \dd v \dd s\\
&=& \Delta_n^{-1}\int_{(i-1)\Delta_n}^{i\Delta_n}
    \int_{0}^{s-t^n_{i-1}} \left[a(X_{t^n_{i-1}};\theta) +
    vR(v,X_{t^n_{i-1}};\theta)\right] \dd v \dd s \\
&=& \Delta_n \frac{1}{2} a(X_{t^n_{i-1}};\theta) + \Delta_n^{-1}
    \int_{(i-1)\Delta_n}^{i\Delta_n} \int_0^{s-t^n_{i-1}} v
    R(v,X_{t^n_{i-1}};\theta) \dd v \dd s
\end{eqnarray*}
and \eqref{eq:xi} follows because the last term equals $\Delta_n^2
R(\Delta_n,x;\theta)$. In fact, we see that the slightly stronger result $\ET(\xi^*_{2,i}
\hspace*{0.05cm} | \hspace*{0.05cm} \F^n_{i-1}) = \Delta_n \frac{1}{2}
a(X_{t^n_{i-1}};\theta) + \Delta_n^2 R(\Delta_n,X_{t^n_{i-1}};\theta)$
holds for this particular choice of $f$.

To show that $\ET((\xi^*_{2,i})^2 \hspace*{0.05cm} | \hspace*{0.05cm}
\F^n_{i-1}) = \Delta_n^2 R(\Delta_n,X_{t^n_{i-1}};\theta)$, we use
that by Jensen's inequality
\begin{equation*}
\abs{A_i}^2 \leq \Delta_n^{-1}\int_{(i-1)\Delta_n}^{i\Delta_n}
\abs{\int_{(i-1)\Delta_n}^s a(X_u;\theta) \dd u}^2 \dd s \leq \sup_{s
  \in [0,\Delta_n]}\abs{\int_{t^n_{i-1}}^{t^n_{i-1}+s} a(X_u;\theta)
  \dd u}^2. 
\end{equation*}
Moreover, for any $t \geq
0$ (again by Jensen's inequality),
\begin{eqnarray*}
\ET\left(\sup_{s \in [0,\Delta_n]} \abs{\int_{t}^{t+s} a(X_u;\theta) \dd u}^2 \given \F_t\right)
&	\leq	& \ET\left(\sup_{s \in [0,\Delta_n]} s \int_{t}^{t+s} |a(X_u;\theta)|^2 \dd u \given \F_t\right) \\
&=& \Delta_n \ET\left(\int_{t}^{t+\Delta_n} |a(X_u;\theta)|^2 \dd u \given \F_t\right).
\end{eqnarray*}
Now by the linear growth of $a(\ccs;\theta)$ (Condition \ref{cond:X}),
$|a(X_u;\theta)|^2 \leq_{C} 1 + |X_t|^2 + |X_u-X_t|^2$, so
\begin{eqnarray*}
\lefteqn{\ET\left(\sup_{s \in [0,\Delta_n]} \abs{\int_{t}^{t+s} a(X_u;\theta)
  \dd u}^2 \given \F_t\right)} \\ 
&\leq_{C}& \Delta_n^2 (1+|X_t|^2)+\Delta_n \int_t^{t+\Delta_n}\ET\left(|X_u-X_t|^2\given\F_t\right) \dd u \\
&\leq_{C}& \Delta_n^2 (1+|X_t|)^{C} + \Delta_n^2 \cc \ET\left(\sup_{v \in [0,\Delta_n]}|X_{t+v}-X_t|^2 \given \F_t\right)
	\leq_{C}	\Delta_n^2 (1+|X_t|)^{C},
\end{eqnarray*}
where (\ref{eq:Esupk}) implies the final inequality. In conclusion,
$\ET\left(\abs{A_i}^2 \given \F^n_{i-1}\right) 
\leq_{C} \Delta_n^2 (1+|X_{t^n_{i-1}}|)^{C}$.
To obtain a similar bound for $|D_i|^2$, we apply the
Burkholder-Davis-Gundy inequality, Jensen's inequality and
(\ref{eq:Esupk})  to obtain that
\begin{eqnarray*}
\ET\left(\abs{D_i}^2 \given \F^n_{i-1}\right)
& =&				 \ET\left(\Delta_n^{-2} \abs{\int_{(i-1)\Delta_n}^{i\Delta_n} \left[b(X_s;\theta)-b(X_{t^n_{i-1}};\theta)\right](i\Delta_n-s) dB_s}^2 \given \F^n_{i-1}\right) \\
& \leq_{C}& \ET\left(\int_{(i-1)\Delta_n}^{i\Delta_n}
               \left[b(X_s;\theta)-b(X_{t^n_{i-1}};\theta)\right]^2
               \dd s \given \F^n_{i-1}\right) \\
& \leq & \Delta_n \ET\left(\sup_{s \in [0,\Delta_n]} \abs{ b(X_{t^n_{i-1}+s};\theta)
  -  b(X_{t^n_{i-1}};\theta)}^2 \given \F^n_{i-1}\right) \\
& \leq_{C}&	 \Delta_n^{2}(1+|X_{t^n_{i-1}}|)^{C}.
\end{eqnarray*}
and, as a consequence,
\begin{equation}\label{eqn:xi}
\ET\left((\xi^*_{2,i})^2 \given \F^n_{i-1}\right) = \Delta_n^2
R(\Delta_n,X_{t^n_{i-1}};\theta), 
\end{equation}

\vspace*{0.2cm}

The extension to arbitrary $f \in \C^4_p(S)$ is based on Taylor
expansions of the general form \eqref{eq:taylor}. First, a Taylor
expansion combined with the Euler-Itô expansion (\ref{eq:YX-identity}), implies that
\begin{eqnarray*}
f(Y_i)
& = & \sum_{j=0}^{2} \frac{1}{j!} \partial_x^j f(X_{t^n_{i-1}}) (Y_i-X_{t^n_{i-1}})^j + \frac{1}{6} \partial_x^3 f(Z^n_i) (Y_i-X_{t^n_{i-1}})^3 \\
& = & f(X_{t^n_{i-1}}) + \Delta_n^{1/2} \partial_x f(X_{t^n_{i-1}}) b(X_{t^n_{i-1}};\theta)\xi_{1,i} + \xi_{2,i}
\end{eqnarray*}
where
\begin{equation}\label{eq:xi1}
\xi_{1,i} = \xi^*_{1,i} = \Delta_n^{-3/2}\int_{(i-1)\Delta_n}^{i\Delta_n} (i\Delta_n-s) dB_s,
\end{equation}
and $\xi_{2,i} = \sum_{k=1}^5 \xi^{(k)}_{2,i}$ with $\xi^{(1)}_{2,i}
=  \partial_x f(X_{t^n_{i-1}}) \xi^*_{2,i}$, $\xi^{(2)}_{2,i}  =
\Delta_n \frac{1}{2} \partial^2_x f(X_{t^n_{i-1}})
b^2(X_{t^n_{i-1}};\theta) (\xi^*_{1,i})^2$, $\xi^{(3)}_{2,i}  =
\frac{1}{2} \partial^2_x f(X_{t^n_{i-1}})(\xi^*_{2,i})^2$, 
$\xi^{(4)}_{2,i}  =  \Delta_n^{1/2} \partial^2_x f(X_{t^n_{i-1}})
b(X_{t^n_{i-1}};\theta)\xi^*_{1,i}\xi^*_{2,i}$ and $\xi^{(5)}_{2,i}  =
\frac{1}{6} \partial_x^3 f(Z^n_i) (Y_i-X_{t^n_{i-1}})^3$. 

Each $\xi^{(k)}_{2,i}$, $k = 1, \ldots , 5$, is measurable
w.r.t.\ $\F^n_i$ so it only remains to show that $\xi_{2,i}$ satisfies
the moment expansions \eqref{eq:exp1-YX} and \eqref{eq:exp2-YX}. 
By applying the previously derived conditional moment expansions $\ET\left(\xi^*_{1,i}
  \given \F^n_{i-1}\right) = 0$, $\ET\left((\xi^*_{1,i})^2 \given
  \F^n_{i-1}\right) = \frac{1}{3}$, (\ref{eq:YX-identity}) and \eqref{eqn:xi}
it follows immediately that
\begin{eqnarray*}
\ET\left(\xi^{(1)}_{2,i} \given \F^n_{i-1}\right) & = & \Delta_n \frac{1}{2} a(X_{t^n_{i-1}};\theta) \partial_x f(X_{t^n_{i-1}}) + \Delta_n^{3/2} R(\Delta_n,X_{t^n_{i-1}};\theta), \\
\ET\left(\xi^{(2)}_{2,i} \given \F^n_{i-1}\right) & = & \Delta_n \frac{1}{6} \partial^2_x f(X_{t^n_{i-1}}) b^2(X_{t^n_{i-1}};\theta), \\
\ET\left(\xi^{(3)}_{2,i} \given \F^n_{i-1}\right) & = & \Delta_n^2 R(\Delta_n,X_{t^n_{i-1}};\theta).
\end{eqnarray*}
Furthermore, by Hölder's inequality,
\begin{equation*}
|\ET\left(\xi^*_{1,i}\xi^*_{2,i} \given \F^n_{i-1}\right)| \leq \ET\left((\xi^*_{1,i})^2 \given \F^n_{i-1}\right)^{1/2} \ET\left((\xi^*_{2,i})^2 \given \F^n_{i-1}\right)^{1/2} = \Delta_n R(\Delta_n,X_{t^n_{i-1}};\theta),
\end{equation*}
implying $\ET\left(\xi^{(4)}_{2,i} \given \F^n_{i-1}\right) = \Delta_n^{3/2}
R(\Delta_n,X_{t^n_{i-1}};\theta)$,  and finally, by Lemma \ref{lem:YX-rem},
$\ET\left(\xi^{(5)}_{2,i} \given \F^n_{i-1}\right) = \Delta_n^{3/2}
R(\Delta_n,X_{t^n_{i-1}};\theta)$. Collecting our observations,
\begin{eqnarray*}
\lefteqn{\ET\left(\xi_{2,i} \given \F^n_{i-1}\right) 
	=	\sum_{k=1}^5 \ET\left(\xi^{(k)}_{2,i} \given \F^n_{i-1}\right)} \\
&	=	& \Delta_n \left(\frac{1}{2} a(X_{t^n_{i-1}};\theta) \partial_x f(X_{t^n_{i-1}}) + \frac{1}{6} b^2(X_{t^n_{i-1}};\theta) \partial^2_x f(X_{t^n_{i-1}})\right) + \Delta_n^{3/2} R(\Delta_n,X_{t^n_{i-1}};\theta) \\
&	=	& \Delta_n \left(\frac{1}{2} \GT f(X_{t^n_{i-1}}) - \frac{1}{12} b^2(X_{t^n_{i-1}};\theta) \partial^2_x f(X_{t^n_{i-1}})\right) + \Delta_n^{3/2} R(\Delta_n,X_{t^n_{i-1}};\theta).
\end{eqnarray*}

To argue that $\ET(\xi^2_{2,i}  \hspace*{0.05cm} |
\hspace*{0.05cm} \F^n_{i-1}) = \Delta_n^2
R(\Delta_n,X_{t^n_{i-1}};\theta)$, we combine a lower order Taylor
expansion with (\ref{eq:YX-identity}) to obtain
\begin{eqnarray*}
f(Y_i)
& = & f(X_{t^n_{i-1}}) + \partial_x f(X_{t^n_{i-1}}) (Y_i-X_{t^n_{i-1}}) + \frac{1}{2} \partial_x^2 f(Z^n_i) (Y_i-X_{t^n_{i-1}})^2 \\
& = & f(X_{t^n_{i-1}}) + \Delta_n^{1/2} \partial_x f(X_{t^n_{i-1}})
      b(X_{t^n_{i-1}};\theta)\xi_{1,i} + \xi_{2,i}, 
\end{eqnarray*}
from which we get an alternative expression for the  the remainder term $\xi_{2,i}$:
\begin{equation*}
\xi_{2,i} = \partial_x f(X_{t^n_{i-1}}) \xi^*_{2,i} + \frac{1}{2} \partial^2_x f(Z^n_i)(Y_i-X_{t^n_{i-1}})^2.
\end{equation*}
This expression implies that
\begin{equation*}
\xi^2_{2,i} \leq_C [\partial_x f(X_{t^n_{i-1}})]^2 (\xi^*_{2,i})^2 + [\partial^2_x f(Z^n_i)]^2 (Y_i-X_{t^n_{i-1}})^4
\end{equation*}
and, by applying \eqref{eqn:xi} and Lemma \ref{lem:YX-rem}, that
$\ET\left(\xi^2_{2,i} \given \F^n_{i-1}\right) = \Delta_n^2
R(\Delta_n,X_{t^n_{i-1}};\theta)$. 

Finally, by the definitions \eqref{eqn:var1} and \eqref{eq:xi1} and the Itô isometry,
\begin{eqnarray*}
\ET(\varepsilon_{1,i}\xi_{1,i})
&=& \ET\left(\Delta_n^{-1/2} \int_{(i-1)\Delta_n}^{i\Delta_n} dB_s  \cc \Delta_n^{-3/2} \int_{(i-1)\Delta_n}^{i\Delta_n} (i\Delta_n-s) dB_s\right) \\
&=& \Delta_n^{-2} \cc \int_{(i-1)\Delta_n}^{i\Delta_n} (i\Delta_n-s) \dd s
= \frac{1}{2}.
\end{eqnarray*}
\end{proof}

\begin{proof}[Proof of Lemma \ref{lem:LLN}]
The Lemma follows from Lemma~3.1 in \cite{ped} if we show that
\begin{equation}\label{eq:YX-LLN}
\frac{1}{n} \sum_{i=1}^n \left[f(Y_i)-f(X_{t^n_{i-1}})\right] = o_{\Pnull}(1).
\end{equation}
By applying the bound \eqref{eq:YX-bound} for conditional
expectations, we obtain
\begin{equation*}
\frac{1}{n} \sum_{i=1}^n \Enull\left(|f(Y_i)-f(X_{t^n_{i-1}})| \given
  \F^n_{i-1}\right) = \Delta_n^{1/2} \frac{1}{n} \sum_{i=1}^n
R(\Delta_n,X_{t^n_{i-1}};\theta_0) = o_{\Pnull}(1), 
\end{equation*}
\begin{equation*}
\frac{1}{n^2} \sum_{i=1}^n \Enull\left(|f(Y_i)-f(X_{t^n_{i-1}})|^2
  \given \F^n_{i-1}\right) = \Delta_n \frac{1}{n^2} \sum_{i=1}^n
R(\Delta_n,X_{t^n_{i-1}};\theta_0) = o_{\Pnull}(1), \hspace*{0.1cm} 
\end{equation*}
from which \eqref{eq:YX-LLN} follows by Lemma~9 in \cite{edc-1993}.
\end{proof}

\begin{proof}[Proof of Lemma \ref{lem:CLT}]
This result follows from Proposition~3.4 in \cite{ped}, if the
following strengthening of (\ref{eq:YX-LLN}) holds
\begin{equation}\label{eq:YX-CLT}
\sqrt{n\Delta_n} \cc \frac{1}{n} \sum_{i=1}^n \left[f(Y_i)-f(X_{t^n_{i-1}})\right] = o_{\Pnull}(1),
\end{equation}
To prove this, note that by Proposition \ref{prop:YX},
\begin{eqnarray*}
\lefteqn{\sqrt{n\Delta_n} \cc \frac{1}{n} \sum_{i=1}^n
  \Enull\left(f(Y_i)-f(X_{t^n_{i-1}}) \given \F^n_{i-1}\right)} \\
& = & \sqrt{n\Delta_n} \cc \frac{1}{n} \sum_{i=1}^n \Enull\left(\xi_{2,i}\given \F^n_{i-1}\right)
 =  \sqrt{n\Delta_n^3} \cc \frac{1}{n} \sum_{i=1}^n R(\Delta_n,X_{t^n_{i-1}};\theta_0)
 =  o_{\Pnull}(1),
\end{eqnarray*}
where we use that $n\Delta_n^3 \to 0$. Moreover, the higher order
bound \eqref{eq:YX-bound} ensures that 
\begin{equation*}
\frac{\Delta_n}{n} \sum_{i=1}^n \Enull\left(|f(Y_i)-f(X_{t^n_{i-1}})|^2 \given \F^n_{i-1}\right) = \frac{\Delta_n^2}{n} \sum_{i=1}^n R(\Delta_n,X_{t^n_{i-1}};\theta_0) = o_{\Pnull}(1),
\end{equation*}
and \eqref{eq:YX-CLT} follows from Lemma~9 in \cite{edc-1993}.
\end{proof}

\begin{proof}[Proof of Theorem \ref{thm:PBEF.s}]
By applying the first order expansion \eqref{eq:G.s} of $\ET f(Y_1)$
together with Lemma \ref{lem:LLN}, we see that
\begin{equation*}
H_n(\theta) =	\frac{1}{n}\sum_{i=1}^n \left[f(Y_i)-\ET f(Y_1)\right] 
= \frac{1}{n}\sum_{i=1}^n \left[f(Y_i) - \mu_\theta(f)\right] +
\Delta_n R(\Delta_n;\theta) \cPnull	 H(\theta),
\end{equation*}
where $H(\theta) = (\mu_0-\mu_\theta)(f)$. Under Condition \ref{cond:G.s},
\begin{equation*}
\partial_\theta H_n(\theta) = -\partial_\theta \ET f(Y_1) = -\partial_\theta\mu_\theta(f) + \Delta_n \partial_\theta R(\Delta_n;\theta) \to -\partial_\theta\mu_\theta(f),
\end{equation*}
and for any compact subset $\M$ of $\Theta$
\begin{equation*}
\sup_{\theta \in \M} |\partial_\theta
H_n(\theta)+\partial_\theta\mu_\theta(f)| = \Delta_n \sup_{\theta \in
  \M} |\partial_\theta R(\Delta_n;\theta)| \leq {C(\M)}\Delta_n \to 0.
\end{equation*}
Because $H(\theta_0)=0$, we have now verified the conditions of
Theorem~2.5 in \cite{jjms}, from which the existence of a
consistent sequence of $G_n$-estimators $(\hat{\theta}_n)$
follows. That the estimator $\hat{\theta}_n$ is unique in any
compact subset $\K \subseteq \Theta$ that contains $\theta_0$ with
$\Pnull$-probability going to one as $n \to \infty$ follows from
Theorem~2.7 in \cite{jjms}, because the identifiability assumption implies
that $H(\theta) \neq 0$ for $\theta \neq \theta_0$.

To establish asymptotic normality, note that \eqref{eq:G.s}, the
additional rate assumption $n\Delta_n^3 \to 0$ and Lemma \ref{lem:CLT} ensure that
\[
\sqrt{n\Delta_n} \cc H_n(\theta_0) = \sqrt{n\Delta_n} \cc
\left(\frac{1}{n} \sum_{i=1}^n f^*(Y_i)\right) + \sqrt{n\Delta_n^3} R(\Delta_n;\theta_0)
\cDnull \N\left(0,V_0(f)\right).
\]
Now (\ref{eq:AN.s}) follows by a standard Taylor expansion argument,
see e.g.\ Theorem 2.11 in \cite{jjms}.
\end{proof}

\begin{proof}[Proof of Lemma \ref{lem:a}]
We break the proof into three steps: \textbf{Step 1:} expand $[\ET
f(Y_1)]^2$, $\ET f^2(Y_1)$ and $\ET\left[f(Y_1)f(Y_2)\right]$ in
powers of $\Delta_n$, \textbf{Step 2:} eliminate $\HT$ from the
expansions, \textbf{Step 3:} calculate expansions of
$\breve{a}_n(\theta)_0$ and $\breve{a}_n(\theta)_1$. 

\paragraph{Step 1}

Using Proposition \ref{prop:YX}, we find that
\begin{eqnarray*}
[\ET f(Y_1)]^2	&=& \mu_\theta(f)^2 + \Delta_n
                    2\mu_\theta(f)\mu_\theta(\HT f) + \Delta_n^{3/2}
                    R(\Delta_n;\theta), \\
\ET f^2(Y_1)	&=& \mu_\theta(f^2) + \Delta_n \mu_\theta(\HT f^2) + \Delta_n^{3/2}
                    R(\Delta_n;\theta). 
\end{eqnarray*}

To expand $\ET[f(Y_1)f(Y_2)]$ we note that by Proposition
\ref{prop:YX} this mixed moment equals
\begin{eqnarray*}
\ET\left[\left(f(X_0) + \Delta_n^{1/2} \partial_x f(X_0) b(X_0;\theta)
  \xi_{1,1}+\xi_{2,1}\right)\left(f(X_{\Delta_n}) + \Delta_n^{1/2}
  \partial_x f(X_{\Delta_n}) b(X_{\Delta_n};\theta)
  \xi_{1,2}+\xi_{2,2}\right)\right], 
\end{eqnarray*}
and then we expand the 9 terms of this expectation individually.

\vspace{2mm}

\noindent
{\it Term 1}:
\begin{equation*}
\ET\left[f(X_0)f(X_{\Delta_n})\right] = \mu_\theta(f^2) + \Delta_n \mu_\theta(f \GT f) + \Delta_n^2 R(\Delta_n;\theta)
\end{equation*}
because by Proposition \ref{prop:XX}
\begin{equation*}
f(X_0) f(X_{\Delta_n}) = f(X_0)^2 +\Delta_n^{1/2}f(X_0)\partial_x
f(X_0) b(X_0;\theta) \varepsilon_{1,1}+f(X_0) \varepsilon_{2,1}.
\end{equation*}

\noindent
{\it Term 2}:
\begin{equation*}
\Delta_n^{1/2} \ET\left[f(X_0)\partial_x f(X_{\Delta_n}) b(X_{\Delta_n};\theta) \xi_{1,2}\right]=0,
\end{equation*}
because $\ET\left(\xi_{1,2} \given \F_{\Delta_n}\right)$ = 0.

\vspace{2mm}

\noindent
{\it Term 3}: By applying the moment expansions \eqref{eq:exp1-YX} and
(\ref{eq:condmomexp}) 
\begin{eqnarray*}
\lefteqn{\ET\left[f(X_0)\xi_{2,2}\right]
= \ET\left[f(X_0) \ET\left(\xi_{2,2} \given \F_{\Delta_n}\right)\right] 
= \Delta_n \ET\left[f(X_0) \HT f(X_{\Delta_n})\right] + \Delta_n^{3/2}
  R(\Delta_n;\theta)} \\ 
&=& \Delta_n \ET\left[f(X_0) \ET\left(\HT f(X_{\Delta_n}) \given
    \F_0\right)\right] + \Delta_n^{3/2} R(\Delta_n;\theta)
= \Delta_n \mu_\theta(f \HT f) + \Delta_n^{3/2} R(\Delta_n;\theta).
\end{eqnarray*}

\noindent
{\it Term 4}: By the  Euler-Itô expansion \eqref{eq:XX},
\begin{eqnarray*}
\lefteqn{\Delta_n^{1/2} \ET\left[\partial_xf(X_0)b(X_0;\theta)\xi_{1,1}f(X_{\Delta_n})\right]} \\
&=& \Delta_n
   \ET\left[[\partial_xf(X_0)b(X_0;\theta)]^2\xi_{1,1}\varepsilon_{1,1}\right]
   + \Delta_n^{1/2}
   \ET\left[\partial_xf(X_0)b(X_0;\theta)\xi_{1,1}\varepsilon_{2,1}\right] \\
   &=& \Delta_n \frac{1}{2} \mu_\theta\left([b(\ccs;\theta)\partial_x
       f]^2\right) + \Delta_n^{3/2} R(\Delta_n;\theta). 
\end{eqnarray*}
The last equality holds because, by \eqref{eq:isometry}, $\ET\left(
  \xi_{1,1}\varepsilon_{1,1} \given \F_0\right)=1/2$, and because by
Hölder's inequality and \eqref{eq:exp2-XX}, and since $\xi_{1,1} \sim
\N(0,1/3)$ and is independent of $\F^n_0$, we see that 
\begin{equation*}
|\ET\left(\xi_{1,1}\varepsilon_{2,1} \given \F_0\right)| \leq \ET\left(\xi_{1,1}^2 \given \F_0\right)^{1/2} \ET\left(\varepsilon_{2,1}^2 \given \F_0\right)^{1/2} = \Delta_n R(\Delta_n,X_0;\theta).
\end{equation*}

\noindent
{\it Term 5}: This term equals
\begin{equation*}
\Delta_n^2 \ET\left[\partial_x f(X_0) b(X_0;\theta) \xi_{1,1} \partial_x f(X_{\Delta_n}) b(X_{\Delta_n};\theta) \ET\left(\xi_{1,2} \given \F_{\Delta_n}\right)\right] = 0,
\end{equation*}
since $\ET\left(\xi_{1,2} \given \F_{\Delta_n}\right) = 0$ by Proposition \ref{prop:YX}.

\vspace{2mm}

\noindent
{\it Term 6}: This term equals
\begin{equation*}
\Delta_n^{1/2} \ET\left[\partial_x f(X_0) b(X_0;\theta)
  \ET\left(\xi_{1,1} \xi_{2,2} \given \F_0\right)\right] =
\Delta_n^{3/2} R(\Delta_n;\theta), 
\end{equation*}
because by Hölder's inequality, $|\ET\left(\xi_{1,1} \xi_{2,2}
  \given \F_0\right)| \leq \ET\left(\xi_{1,1}^2 \given
  \F_0\right)^{1/2} \ET\left(\xi_{2,2}^2 \given \F_0\right)^{1/2}$, 
and by Proposition \ref{prop:YX}, $\ET\left(\xi_{1,1}^2 \given
  \F_0\right) = 1/3$ and
\begin{equation*}
\ET\left(\xi_{2,2}^2 \given \F_0\right)=\ET\left[\ET\left(\xi_{2,2}^2 \given \F_{\Delta_n}\right) \given \F_0\right]=\ET\left[\Delta_n^2 R(\Delta_n,X_{\Delta_n};\theta) \given \F_0\right]=\Delta_n^2 R(\Delta_n,X_0;\theta).
\end{equation*}

\noindent
{\it Term 7}: By the Euler-Itô expansion \eqref{eq:XX},
\begin{eqnarray*}
  \lefteqn{\ET\left[f(X_{\Delta_n})\xi_{2,1}\right]} \\
  &=& \ET\left[f(X_0)\ET\left(\xi_{2,1} \given \F_0\right)\right] +
  \Delta_n^{1/2} \ET\left[\partial_x f(X_0)b(X_0;\theta)
  \ET\left(\varepsilon_{1,1}\xi_{2,1} \given \F_0\right)\right] +
      \ET(\varepsilon_{2,1}\xi_{2,1}) \\
  &=& \Delta_n \mu_\theta(f \HT f) + \Delta_n^{3/2} R(\Delta_n;\theta)
  \end{eqnarray*}
  where the last equality holds because, by Proposition \ref{prop:YX},
  $\ET\left(\xi_{2,1} \given \F_0\right) = \Delta_n \HT f(X_0) +
  \Delta_n^{3/2} R(\Delta_n, X_0;\theta)$, and by Hölder's inequality
  and Propositions \ref{prop:XX} and \ref{prop:YX},
  $\ET\left(\varepsilon_{1,1}\xi_{2,1} \given \F_0\right) = \Delta_n
  R(\Delta_n,X_0;\theta)$ and
  $\ET(\varepsilon_{2,1}\xi_{2,1})=\Delta_n^2 R(\Delta_n;\theta)$. 

  \vspace{2mm}

\noindent
{\it Term 8}: By Proposition \ref{prop:YX}, this term equals
\begin{equation*}
\ET\left[\xi_{2,1} \partial_x f(X_{\Delta_n}) b(X_{\Delta_n};\theta) \ET\left(\xi_{1,2} \given \F_{\Delta_n}\right)\right] = 0.
\end{equation*}

\noindent
{\it Term 9}: By combining Hölder's inequality and (\ref{eq:exp2-YX}), we obtain
\begin{equation*}
|\ET\left(\xi_{2,1}\xi_{2,2}\right)| \leq \ET\left[\ET\left(\xi_{2,1}^2 \given \F_0\right)\right]^{1/2} \ET\left[\ET\left(\xi_{2,2}^2 \given \F_{\Delta_n}\right)\right]^{1/2} = \Delta_n^2 R(\Delta_n;\theta).
\end{equation*}

\vspace*{0.2cm}

Finally, we add the expansions of the nine terms and conclude that
\begin{eqnarray*}
\lefteqn{\ET f(Y_1)f(Y_2) =} \\
&& \mu_\theta(f^2) + \Delta_n\left(\mu_\theta(f \GT f) + 2\mu_\theta(f
   \HT f) + \frac{1}{2} \mu_\theta\left([b(\ccs;\theta)\partial_x
   f]^2\right)\right) + \Delta_n^{3/2} R(\Delta_n;\theta). 
\end{eqnarray*}

\paragraph{Step 2}

To eliminate $\HT$ from the expansions of $[\ET f(Y_1)]^2$, $\ET
f^2(Y_1)$ and $\ET\left[f(Y_1)f(Y_2)\right]$, we rewite
$\mu_\theta(\HT f)$, $\mu_\theta(f \HT f)$ and $\mu_\theta(\HT f^2)$
using the definition of $\HT$, \eqref{eq:HT}, and that $\mu_\theta(\GT f)=0$
for all $f \in \D_{\A_\theta}$; see e.g. \cite{mmp}. It follows
immediately that
\begin{eqnarray*}
\mu_\theta(\HT f) &=& -\frac{1}{12}
 \mu_\theta\left(b^2(\ccs;\theta)\partial^2_x f \right) \\ 
\mu_\theta(f \HT f) &=& \frac{1}{2} \mu_\theta(f \GT f) - \frac{1}{12}
\mu_\theta\left(f b^2(\ccs;\theta)\partial^2_x f\right).
\end{eqnarray*}
Moreover, since $\partial_x f^2 = 2 f \partial_x f$ and $\partial^2_x f^2 = 2\left[f\partial^2_x f+(\partial_x f)^2\right]$,
\begin{eqnarray*}
\HT f^2(x)
&=& \frac{1}{2} a(x;\theta) \partial_x f^2(x) + \frac{1}{6} b^2(x;\theta) \partial_x^2 f^2(x) \\
&=& f(x) a(x;\theta)\partial_x f(x) + \frac{1}{3} f(x) b^2(x;\theta) \partial_x^2 f(x) + \frac{1}{3} [b(x;\theta)\partial_x f(x)]^2 \\
&=& f(x) \GT f(x) - \frac{1}{6} f(x) b^2(x;\theta) \partial_x^2 f(x) + \frac{1}{3} [b(x;\theta)\partial_x f(x)]^2,
\end{eqnarray*}
which shows that
\begin{equation*}
\mu_\theta(\HT f^2) = \mu_\theta(f \GT f) - \frac{1}{6}
\mu_\theta\left(f b^2(\ccs;\theta) \partial_x^2 f\right) + \frac{1}{3}
\mu_\theta\left([b(\ccs;\theta)\partial_x f]^2\right). 
\end{equation*}

Thus
\begin{eqnarray}
  \label{eqn:M1}
 [\ET f(Y_1)]^2 &=& \mu_\theta(f)^2 + \Delta_n M_0(\theta) + \Delta_n^{3/2}
                    R(\Delta_n;\theta) \\
  \ET f^2(Y_1) &=& \mu_\theta(f^2) + \Delta_n M_1(\theta) + \Delta_n^{3/2}
                    R(\Delta_n;\theta) \label{eqn:M2} \\
\ET f(Y_1)f(Y_2) &=& \mu_\theta(f^2) + \Delta_n M_2(\theta) + \Delta_n^{3/2}
                    R(\Delta_n;\theta), \label{eqn:M3}
\end{eqnarray}
where
\begin{eqnarray*}
M_0(\theta) &=& - \frac{1}{6}
\mu_\theta(f)\mu_\theta\left(b^2(\ccs;\theta)\partial^2_x  f\right) \\
M_1(\theta) &=& \mu_\theta(f \GT f) -
                \frac{1}{6}\mu_\theta\left(f
                b^2(\ccs;\theta)\partial^2_x
                f\right)+\frac{1}{3}\mu_\theta\left([b(\ccs;\theta)\partial_x
                f]^2\right) \\                  
M_2(\theta) &=& 2\mu_\theta(f \GT f) -
                \frac{1}{6}\mu_\theta\left(f
                b^2(\ccs;\theta)\partial^2_x f\right) + \frac{1}{2}
                \mu_\theta\left([b(\ccs;\theta)\partial_x
                f]^2\right). 
\end{eqnarray*}

\paragraph{Step 3}

From the moment expansions \eqref{eqn:M1}-\eqref{eqn:M3}, it follows that
\begin{eqnarray*}
\breve{a}_n(\theta)_1
&=& \frac{\ET f(Y_1)f(Y_2) - [\ET f(Y_1)]^2}{\VT f(Y_1)} \nonumber \\
&=& \frac{1 + \Delta_n \VT f(X_0)^{-1} (M_2(\theta) - M_0(\theta)) +
    \Delta_n^{3/2} R(\Delta_n;\theta)}{1 + \Delta_n \VT f(X_0)^{-1}
    (M_1(\theta) - M_0(\theta)) + \Delta_n^{3/2} R(\Delta_n;\theta)},
\end{eqnarray*}
and since $1/(1+x) = 1-x+O(x^2)$, we obtain the expansion
\begin{eqnarray}\label{eq:a1}
\breve{a}_n(\theta)_1
&=& 1 + \Delta_n \VT f(X_0)^{-1} \left[M_2(\theta)-M_1(\theta)\right] + \Delta_n^{3/2} R(\Delta_n;\theta) \nonumber \\
&=& 1 + \Delta_n K_f(\theta) + \Delta_n^{3/2} R(\Delta_n;\theta),
\end{eqnarray}
where $K_f(\theta)$ is given by \eqref{eq:K}.

Finally, since by (\ref{eq:YX})  $\ET f(Y_1) = \mu_\theta(f) + \Delta_n R(\Delta_n;\theta)$,
\eqref{eq:a1} implies that 
\begin{equation*}
\breve{a}_n(\theta)_0 = \ET f(Y_1)\left(1-\breve{a}_n(\theta)_1\right) = -\Delta_n K_f(\theta)\mu_\theta(f) + \Delta_n^{3/2} R(\Delta_n;\theta).
\end{equation*}
\end{proof}

\begin{proof}[Proof of Lemma \ref{lem:PBEF.1}]
  Define
\begin{eqnarray}
\label{eqn:g1}	g_1(\Delta_n,Y_i,Y_{i-1};\theta) &=& f(Y_i)-\breve{a}_n(\theta)_0-\breve{a}_n(\theta)_1f(Y_{i-1}), \\
						g_2(\Delta_n,Y_i,Y_{i-1};\theta) &=& f(Y_{i-1})\left[f(Y_i)-\breve{a}_n(\theta)_0-\breve{a}_n(\theta)_1f(Y_{i-1})\right], \nonumber
\end{eqnarray}
and
$H_n(\theta) = (n\Delta_n)^{-1}G_n(\theta)
= (n\Delta_n)^{-1} \sum_{i=2}^n g(\Delta_n,Y_i,Y_{i-1};\theta)$,
where $g=(g_1,g_2)^T$. 

By the expansion \eqref{eq:a} 
\begin{equation}\label{eqn:g1-exp-Y}
g_1(\Delta_n,Y_i,Y_{i-1};\theta) = f(Y_i)-f(Y_{i-1}) + \Delta_n K_f(\theta)\left[\mu_\theta(f)-f(Y_{i-1})\right] + \Delta_n^{3/2} R(\Delta_n,Y_{i-1};\theta),
\end{equation}
and, hence, by the law of large numbers for integrated diffusions (Lemma \ref{lem:LLN}),
\begin{eqnarray*}
\lefteqn{\frac{1}{n\Delta_n} \sum_{i=2}^n g_1(\Delta_n,Y_i,Y_{i-1};\theta)
	=\frac{1}{n\Delta_n} \left[f(Y_n)-f(Y_1)\right]+ \frac{1}{n} \sum_{i=2}^n 
  K_f(\theta)\left[\mu_\theta(f)-f(Y_{i-1})\right]} \\
  && \hspace{5cm} + \Delta_n^{1/2}\frac{1}{n}\sum_{i=2}^nR(\Delta_n,Y_{i-1};\theta)  
\cPnull	K_f(\theta) (\mu_\theta-\mu_0)(f).
\end{eqnarray*}

The second coordinate of $H_n(\theta)$ requires a considerably longer
proof, because the contribution from the first term is not asymptotically
negligible. To shorten the notation, we define
$\E^n_i	= \partial_x f(X_{t^n_{i-1}}) b(X_{t^n_{i-1}};\theta_0)
\varepsilon_{1,i}$ and $\Xi^n_i	= \partial_x f(X_{t^n_{i-1}})
b(X_{t^n_{i-1}};\theta_0) \xi_{1,i}$ and write the expansions
(\ref{eq:XX}) and (\ref{eq:YX})  under the true probability measure $\Pnull$ as
\begin{eqnarray}
\label{eqn:ito.E} f(X_{t^n_i})	&=& f(X_{t^n_{i-1}}) + \Delta_n^{1/2} \E^n_i + \varepsilon_{2,i}, \hspace*{0.8cm} \\
\label{eqn:ito.X} f(Y_i)			&=& f(X_{t^n_{i-1}}) + \Delta_n^{1/2} \Xi^n_i + \xi_{2,i}.
\end{eqnarray}

First note that by (\ref{eqn:ito.X})
\begin{equation}\label{eqn:g2andg1}
g_2(\Delta_n,Y_i,Y_{i-1};\theta) = \left(f(X_{t^n_{i-2}}) + \Delta_n^{1/2} \Xi^n_{i-1} + \xi_{2,i-1}\right)g_1(\Delta_n,Y_i,Y_{i-1};\theta).
\end{equation}
By inserting \eqref{eqn:ito.X} into \eqref{eqn:g1} and
applying the expansion (\ref{eq:a}) of $\breve{a}_n(\theta)$, we find
that 
\begin{eqnarray}\label{eqn:g1-exp-X}
\lefteqn{g_1(\Delta_n,Y_i,Y_{i-1};\theta) = f(X_{t^n_{i-1}})-f(X_{t^n_{i-2}})} \\
&+& \Delta_n K_f(\theta)\left[\mu_\theta(f)-f(X_{t^n_{i-2}})\right] +
    \Delta_n^{1/2} \left(\Xi^n_i-\Xi^n_{i-1}\right) +
    \mathcal{R}_1(\Delta_n,(X_s)_{s \in
    [t^n_{i-2},t^n_i]},\theta_0;\theta), \nonumber
\end{eqnarray}
where the remainder term $\mathcal{R}_1$ has the form
\begin{eqnarray}\label{eqn:R1}
\lefteqn{\mathcal{R}_1(\Delta_n,(X_s)_{s \in [t^n_{i-2},t^n_i]},\theta_0;\theta)}  \\
&=& \left(\xi_{2,i}-\xi_{2,i-1}\right) -
  \Delta_n^{3/2}K_f(\theta)\Xi^n_{i-1} - \Delta_n
  K_f(\theta)\xi_{2,i-1} + \Delta_n^{3/2} R(\Delta_n,Y_{i-1};\theta). \nonumber
\end{eqnarray}
Using \eqref{eqn:ito.E} and inserting the definitions of $\varepsilon_{2,i}$ and
$\xi_{2,i}$, \eqref{eqn:var1} and \eqref{eqn:var2}, we obtain
\begin{equation}\label{eq:ff_exp}
f(X_{t^n_{i-1}}) - f(X_{t^n_{i-2}}) = \Delta_n^{1/2} \E^n_{i-1} + \varepsilon_{2,i-1} = \Delta_n \Gnull f(X_{t^n_{i-2}}) + A_{i-1}(\theta_0) + M_{i-1}(\theta_0),
\end{equation}
where
\begin{eqnarray*}
A_i(\theta)	&=& \int_{(i-1)\Delta_n}^{i\Delta_n} \left[\GT f(X_s) - \GT f(X_{t^n_{i-1}})\right] \dd s, \\
M_i(\theta)	&=& \int_{(i-1)\Delta_n}^{i\Delta_n} \partial_x f(X_s)b(X_s;\theta) dB_s.
\end{eqnarray*}

Now, using \eqref{eqn:g2andg1},\eqref {eqn:g1-exp-X} and
(\eqref{eq:ff_exp}), we obtain the $\Delta$-expansion 
\begin{equation*}
g_2(\Delta_n,Y_i,Y_{i-1};\theta) = \sum_{k=1}^3 g_2^{(k)}(\Delta_n,Y_i,Y_{i-1};\theta),
\end{equation*}
where 
\begin{eqnarray*}
\lefteqn{g_2^{(1)}(\Delta_n,Y_i,Y_{i-1};\theta) = f(X_{t^n_{i-2}}) \cc
  g_1(\Delta_n,Y_i,Y_{i-1};\theta)} \\ 
& = & \Delta_n f(X_{t^n_{i-2}}) \Gnull f(X_{t^n_{i-2}}) + f(X_{t^n_{i-2}})M_{i-1}(\theta_0) \\
&&  \hspace{5mm} + \ \Delta_n K_f(\theta) f(X_{t^n_{i-2}}) \left[\mu_\theta(f)-f(X_{t^n_{i-2}})\right] + \, \mathcal{R}^{(1)}_2(\Delta_n,(X_s)_{s \in
   [t^n_{i-2},t^n_i]},\theta_0;\theta), 
\end{eqnarray*}
\begin{eqnarray*}
			g_2^{(2)}(\Delta_n,Y_i,Y_{i-1};\theta)
&	=	&	\Delta_n^{1/2} \cc \Xi^n_{i-1} \cc g_1(\Delta_n,Y_i,Y_{i-1};\theta) \\
&	=	&	\Delta_n \left(\E^n_{i-1}-\Xi^n_{i-1}\right)\Xi^n_{i-1} + \mathcal{R}^{(2)}_2(\Delta_n,(X_s)_{s \in [t^n_{i-2},t^n_i]},\theta_0;\theta)
\end{eqnarray*}
and
\begin{equation*}
g_2^{(3)}(\Delta_n,Y_i,Y_{i-1};\theta) = \xi_{2,i-1} \cc g_1(\Delta_n,Y_i,Y_{i-1};\theta) = \mathcal{R}^{(3)}_2(\Delta_n,(X_s)_{s \in [t^n_{i-2},t^n_i]},\theta_0;\theta)
\end{equation*}
with
\begin{eqnarray*}
\lefteqn{\mathcal{R}^{(1)}_2(\Delta_n,(X_s)_{s \in [t^n_{i-2},t^n_i]},\theta_0;\theta) =} \\
&& f(X_{t^n_{i-2}})A_{i-1}(\theta_0) + \Delta_n^{1/2}
   f(X_{t^n_{i-2}})\left(\Xi^n_i-\Xi^n_{i-1}\right) + f(X_{t^n_{i-2}})
   \cc \mathcal{R}_1(\Delta_n,(X_s)_{s \in
   [t^n_{i-2},t^n_i]},\theta_0;\theta)
\end{eqnarray*}
and
\begin{eqnarray*}
\lefteqn{\mathcal{R}^{(2)}_2(\Delta_n,(X_s)_{s \in [t^n_{i-2},t^n_i]},\theta_0;\theta) = \Delta_n^{1/2} \Xi^n_{i-1} \varepsilon_{2,i-1}} \\
&+& \Delta_n^{3/2} \Xi^n_{i-1} K_f(\theta) \left[\mu_\theta(f)-f(X_{t^n_{i-2}})\right] + \Delta_n\Xi^n_i\Xi^n_{i-1} + \Delta_n^{1/2} \Xi^n_{i-1} \mathcal{R}_1(\Delta_n,(X_s)_{s \in [t^n_{i-2},t^n_i]},\theta_0;\theta).
\end{eqnarray*}
Collecting the terms,
\begin{eqnarray}\label{eqn:g2-exp-X}
\lefteqn{g_2(\Delta_n,Y_i,Y_{i-1};\theta) 	= \sum_{k=1}^3 g_2^{(k)}(\Delta_n,Y_i,Y_{i-1};\theta)} \\
&	=	&	\Delta_n f(X_{t^n_{i-2}}) \Gnull f(X_{t^n_{i-2}}) + f(X_{t^n_{i-2}})M_{i-1}(\theta_0) + \Delta_n K_f(\theta) f(X_{t^n_{i-2}}) \left[\mu_\theta(f)-f(X_{t^n_{i-2}})\right] \nonumber \\
&&	\hspace{7mm} + \, \Delta_n
            \left(\E^n_{i-1}-\Xi^n_{i-1}\right)\Xi^n_{i-1} +
            \mathcal{R}_2(\Delta_n,(X_s)_{s \in
            [t^n_{i-2},t^n_i]},\theta_0;\theta), \nonumber 
\end{eqnarray}
where the remainder term is
\begin{equation*}
\mathcal{R}_2(\Delta_n,(X_s)_{s \in [t^n_{i-2},t^n_i]},\theta_0;\theta) = \sum_{k=1}^3 \mathcal{R}^{(k)}_2(\Delta_n,(X_s)_{s \in [t^n_{i-2},t^n_i]},\theta_0;\theta).
\end{equation*}

Now, tedious reasoning based on Lemma~9 in \cite{edc-1993} shows that
\begin{equation}\label{eq:AN.R2.1}
\frac{1}{n\Delta_n} \sum_{i=2}^n \mathcal{R}_2(\Delta_n,(X_s)_{s \in [t^n_{i-2},t^n_i]},\theta_0;\theta) = o_{\Pnull}(1).
\end{equation}
Under the additional rate assumption $n\Delta_n^2 \to 0$, it can in a
similar way be proved that
\begin{equation}\label{eq:AN.R2.2}
\frac{1}{\sqrt{n\Delta_n}} \sum_{i=2}^n \mathcal{R}_2(\Delta_n,(X_s)_{s \in [t^n_{i-2},t^n_i]},\theta_0;\theta) = o_{\Pnull}(1). 
\end{equation}
The latter result is not needed in this proof, but it is necessary to show
asymptotic normality in the proof of  Theorem \ref{thm:PBEF.1}, so we
state it here for convenience. To see that the strong rate assumption
$n\Delta_n^2 \to 0$ is necessary to obtain \eqref{eq:AN.R2.2}, we can,
e.g., consider the last term in \eqref{eqn:R1}:
\begin{equation*}
\frac{1}{\sqrt{n\Delta_n}} \sum_{i=2}^n \Delta_n^{3/2} R(\Delta_n,Y_{i-1};\theta) = \sqrt{n\Delta_n^2} \cc \frac{1}{n} \sum_{i=2}^n R(\Delta_n,Y_{i-1};\theta).
\end{equation*}
As the proofs of \eqref{eq:AN.R2.1} and \eqref{eq:AN.R2.2} are both
very long and not particularly enlightening, they are omitted.

To determine the limit in probability of the second coordinate of
$H_n(\theta)$, we consider each term in \eqref{eqn:g2-exp-X}
separately. By the ergodic theorem, see e.g.\  Lemma~3.1 in \cite{ped}, 
\begin{equation*}
\frac{1}{n} \sum_{i=2}^n f(X_{t^n_{i-2}}) \Gnull f(X_{t^n_{i-2}}) \cPnull \mu_0(f \Gnull f)
\end{equation*}
and
\begin{equation*}
\frac{1}{n} \sum_{i=2}^n K_f(\theta) f(X_{t^n_{i-2}}) \left[\mu_\theta(f)-f(X_{t^n_{i-2}})\right] \cPnull K_f(\theta)\left[\mu_0(f)\mu_\theta(f)-\mu_0(f^2)\right].
\end{equation*}
Furthermore, by definitions of $\E^n_i$ and $\Xi^n_i$,
\begin{equation*}
\frac{1}{n} \sum_{i=1}^n \Enull\left(\left(\E^n_i-\Xi^n_i\right)\Xi^n_i \given \F^n_{i-1}\right) = \frac{1}{n}\sum_{i=1}^n [\partial_x f(X_{t^n_{i-1}}) b(X_{t^n_{i-1}};\theta_0)]^2
\Enull\left((\varepsilon_{1,i}-\xi_{1,i})\xi_{1,i} \given \F^n_{i-1}\right),
\end{equation*}
and since $\xi_{1,i} \sim \N(0,1/3)$, \eqref{eq:isometry} implies that
\begin{equation}\label{eq:eps_xi}
\Enull\left((\varepsilon_{1,i}-\xi_{1,i})\xi_{1,i} \given \F^n_{i-1}\right) = \Enull\left((\varepsilon_{1,i}-\xi_{1,i})\xi_{1,i}\right) = \frac{1}{6},
\end{equation}
so
\begin{equation*}
\frac{1}{n} \sum_{i=1}^n \Enull\left(\left(\E^n_i-\Xi^n_i\right)\Xi^n_i \given \F^n_{i-1}\right) \cPnull \frac{1}{6} \mu_0\left([b(\ccs;\theta_0)\partial_x f]^2\right).
\end{equation*}
Finally, since by the definitions of $\varepsilon_{1,i}$ in
\eqref{eqn:var1} and of $\xi_{1,i}$ in \eqref{eq:xi1} the difference
$\varepsilon_{1,i}-\xi_{1,i}$ is Gaussian, Hölder's inequality
and the ergodic theorem imply that
\begin{eqnarray*}
\lefteqn{\frac{1}{n^2} \sum_{i=1}^n \Enull\left((\E^n_i - \Xi^n_i)^2 (\Xi^n_i)^2 \given \F^n_{i-1}\right)} \\
&=&	\frac{1}{n^2} \sum_{i=1}^n [\partial_x f(X_{t^n_{i-1}})
    b(X_{t^n_{i-1}};\theta_0)]^4 \cc
    \Enull\left(\left(\varepsilon_{1,i}-\xi_{1,i}\right)^2 \xi^2_{1,i}
    \given \F^n_{i-1}\right) = o_{\Pnull}(1).
\end{eqnarray*}
Therefore, by Lemma~9 in \cite{edc-1993}
\begin{equation*}
\frac{1}{n} \sum_{i=1}^n \left(\E^n_i-\Xi^n_i\right)\Xi^n_i \cPnull \frac{1}{6} \mu_0\left([b(\ccs;\theta_0)\partial_x f]^2\right).
\end{equation*}

By the similar arguments,
\begin{equation*}
\frac{1}{n\Delta_n} \sum_{i=1}^n f(X_{t^n_{i-1}})M_i(\theta_0) = o_{\Pnull}(1),
\end{equation*}
where we use that $\Enull\left(M_i(\theta_0) \given \F^n_{i-1}\right)=0$.
Moreover, we use that, with $h(x) = \partial_x f(x) b(x;\theta_0)$, and
using the conditional Itô isometry, Tonelli's theorem and (\ref{eq:condmomexp}),
\begin{eqnarray*}
\lefteqn{\Enull\left(M^2_i(\theta_0) \given \F^n_{i-1}\right)
= \Enull\left(\int_{(i-1)\Delta_n}^{i\Delta_n} h^2(X_s) \dd s \given
    \F^n_{i-1}\right) = \int_{(i-1)\Delta_n}^{i\Delta_n}
    \Enull\left(h^2(X_s) \given \F^n_{i-1}\right) \dd s} \\ 
&	=	& \int_0^{\Delta_n} \left[h^2(X_{t^n_{i-1}}) + u \cc R(u,X_{t^n_{i-1}};\theta_0)\right] \dd u = \Delta_n h^2(X_{t^n_{i-1}}) + \Delta_n^2 R(\Delta_n,X_{t^n_{i-1}};\theta_0)
\end{eqnarray*}
and, therefore,
\begin{eqnarray*}
\lefteqn{\frac{1}{n^2\Delta_n^2} \sum_{i=1}^n \Enull\left(f^2(X_{t^n_{i-1}})M^2_i(\theta_0) \given \F^n_{i-1}\right)} \\
&	=	&	\frac{1}{n\Delta_n}\frac{1}{n} \sum_{i=1}^n f^2(X_{t^n_{i-1}}) h^2(X_{t^n_{i-1}}) + \frac{1}{n^2} \sum_{i=1}^n R(\Delta_n,X_{t^n_{i-1}};\theta_0)
=	o_{\Pnull}(1).
\end{eqnarray*}
Gathering our observations, we have verified \eqref{eq:gamma}. \\

To identify the limit of $\partial_{\theta^T} H_n(\theta)$, we write
\begin{equation*}
H_n(\theta) = \frac{1}{n\Delta_n}\sum_{i=2}^n Z_{i-1}\left[f(Y_i)-Z_{i-1}^T \breve{a}_n(\theta)\right],
\end{equation*}
where $Z_{i-1} = (1,f(Y_{i-1}))^T$, which implies that
\begin{equation*}
\partial_{\theta^T} H_n(\theta) = - \frac{1}{n\Delta_n}\sum_{i=2}^n Z_{i-1}Z_{i-1}^T \partial_{\theta^T} \breve{a}_n(\theta) = Z_n(f)A_n(\theta),
\end{equation*}
with $Z_n(f) := \frac{1}{n} \sum_{i=2}^n Z_{i-1}Z_{i-1}^T$ and
$A_n(\theta) := -\Delta_n^{-1} \partial_{\theta^T}
\breve{a}_n(\theta)$. By Lemma \ref{lem:LLN}, 
\begin{equation*}
Z_n(f) \cPnull Z(f) =: \mtwo{1}{\mu_0(f)}{\mu_0(f)}{\mu_0(f^2)},
\end{equation*}
and applying the expansion \eqref{eq:a} of $\breve{a}_n(\theta)$, we
see that
\begin{equation*}
A_n(\theta) =
\partial_{\theta^T}\vtwo{K_f(\theta)\mu_\theta(f)}{-K_f(\theta)} +
\Delta_n^{1/2} \partial_{\theta^T}R(\Delta_n;\theta) \to
\vtwo{\partial_{\theta^T} [K_f(\theta)\mu_\theta(f)]}{-\partial_{\theta^T} K_f(\theta)} =:
A(\theta). 
\end{equation*}
Hence, it follows that $\partial_{\theta^T} H_n(\theta) \cPnull Z(f)A(\theta)$.
To argue that under Condition \ref{cond:G.1}, the convergence is
uniform over any compact subset $\M$ of $\Theta$, note that
\begin{eqnarray*}
\sup_{\theta \in \M} \norm{\partial_{\theta^T} H_n(\theta) - Z(f)A(\theta)}
& \leq	& \sup_{\theta \in \M} \norm{Z_n(f)[A_n(\theta)-A(\theta)]} +
         \sup_{\theta \in \M}  \norm{[Z_n(f)-Z(f)]A(\theta)} \\
 & \leq	&  \norm{Z_n(f)} \sup_{\theta \in
          \M}\norm{A_n(\theta)-A(\theta)} + \norm{Z_n(f)-Z(f)}
          \sup_{\theta \in \M}\norm{A(\theta)}. 
\end{eqnarray*}
Therefore, \eqref{eq:W} follows because $\theta \mapsto A(\theta)$ and
$\norm{\ccs}$ are continuous, and because
\begin{equation*}
\sup_{\theta \in \M}\norm{A_n(\theta)-A(\theta)} = \Delta_n^{1/2}
\sup_{\theta \in \M} \norm{\partial_{\theta^T} R(\Delta_n;\theta)}
\leq_{C(\M)} \Delta_n^{1/2} \to 0. 
\end{equation*}
\end{proof}

\begin{proof}[Proof of Theorem \ref{thm:PBEF.1}]
We use the notation introduced in the proof of Lemma
\ref{lem:PBEF.1}. Because $\gamma(\theta_0,\theta_0) = 0$, the
existence of a consistent sequence of $G_n$-estimators
$(\hat{\theta}_n)$ follows from Lemma \ref{lem:PBEF.1} and 
Theorem 2.5 in \cite{jjms}.  The eventual uniqueness in $\K$ follows from
Lemma \ref{lem:PBEF.1}  and Theorem 2.7 in the same paper.

Asymptotic normality of $\hat{\theta}_n$ follows by a standard
Taylor expansion argument (see e.g.\ Theorem 2.11 in \cite{jjms}) once
we have established that
\begin{equation}\label{eq:AN.H}
\sqrt{n\Delta_n} \cc H_n(\theta_0) \cDnull \N_2(0,V_0(f)).
\end{equation}

From the expansion of $g_1(\Delta_n,Y_i,Y_{i-1};\theta)$ in
\eqref{eqn:g1-exp-Y}, it follows that 
\begin{eqnarray*}
\lefteqn{\frac{1}{\sqrt{n\Delta_n}} \sum_{i=2}^n g_1(\Delta_n,Y_i,Y_{i-1};\theta_0)} \\
&	=	&	\frac{1}{\sqrt{n\Delta_n}} \left[f(Y_n)-f(Y_1)\right] + \sqrt{n\Delta_n}\left(\frac{1}{n} \sum_{i=2}^n f^*_1(Y_{i-1})\right) + \sqrt{n\Delta_n^2} \cc \frac{1}{n} \sum_{i=2}^n R(\Delta_n,Y_{i-1};\theta_0) \\
&	=	&	\sqrt{n\Delta_n}\left(\frac{1}{n} \sum_{i=2}^n
            f^*_1(Y_{i-1})\right) + o_{\Pnull}(1) \cDnull
            \N\left(0,\mu_0\left([\partial_x
            U_0(f^*_1)b(\ccs;\theta_0)]^2\right)\right), 
\end{eqnarray*}
where the convergence in law holds by Lemma \ref{lem:CLT} because
$f^*_1 \in \mathscr{H}_0$. \\ 

Our proof that
\begin{equation}\label{eq:g2}
\frac{1}{\sqrt{n\Delta_n}} \sum_{i=2}^n g_2(\Delta_n,Y_i,Y_{i-1};\theta_0) \cDnull \N\left(0,\mu_0\left(\left[\partial_x \U_0(f^*_2) + f \partial_x f \right]^2 b^2(\ccs;\theta_0)\right)\right)
\end{equation}
is based on the expansion of $g_2$ given by
\eqref{eqn:g2-exp-X}  and the observation that
\begin{equation}\label{eq:g2-equivalence}
\frac{1}{\sqrt{n\Delta_n}} \sum_{i=2}^n \Delta_n \left(\E^n_{i-1}-\Xi^n_{i-1}\right)\Xi^n_{i-1} =
\frac{1}{6} \frac{1}{\sqrt{n\Delta_n}} \sum_{i=2}^n \Delta_n [\partial_x f(X_{t^n_{i-2}})b(X_{t^n_{i-2}};\theta_0)]^2 + o_{\Pnull}(1),
\end{equation}
which follows from Lemma~9 in \cite{edc-1993} using that
$\left(\E^n_{i-1}-\Xi^n_{i-1}\right)\Xi^n_{i-1}$ $= [\partial_x
f(X_{t^n_{i-2}})b(X_{t^n_{i-2}};\theta_0)]^2
\left(\varepsilon_{1,i-1}-\xi_{1,i-1}\right)\xi_{1,i-1}$ 
and that $\Enull\left(\left(\E^n_{i-1}-\Xi^n_{i-1}\right)\Xi^n_{i-1} \given
  \F^n_{i-2}\right)=$ $\frac{1}{6} [\partial_x
f(X_{t^n_{i-2}})b(X_{t^n_{i-2}};\theta_0)]^2$, see (\ref{eq:eps_xi}). 

\vspace*{0.2cm}

Combining \eqref{eq:g2-equivalence},  \eqref{eqn:g2-exp-X} and the result
\eqref{eq:AN.R2.2} that the term involving the remainder term  vanishes, we see that
\begin{eqnarray}\label{eq:g2-decomposition}
\lefteqn{\frac{1}{\sqrt{n\Delta_n}} \sum_{i=2}^n g_2(\Delta_n,Y_i,Y_{i-1};\theta_0)} \\
& =	&	\sqrt{n\Delta_n}\left(\frac{1}{n} \sum_{i=1}^n
      f_2^*(X_{t^n_{i-1}})\right) + \frac{1}{\sqrt{n\Delta_n}}
      \sum_{i=1}^n f(X_{t^n_{i-1}})M_i(\theta_0) + o_{\Pnull}(1). \nonumber
\end{eqnarray}
To gather the non-negligible terms in \eqref{eq:g2-decomposition}, we
initially observe that 
\begin{eqnarray}\label{eqn:negligible}
\lefteqn{\frac{1}{\sqrt{n\Delta_n}} \int_0^{n\Delta_n} f_2^*(X_s) \dd s  
	=	 \sqrt{n\Delta_n} \left(\frac{1}{n}\sum_{i=1}^n
            f_2^*(X_{t^n_{i-1}})\right)} \\  &+& \frac{1}{\sqrt{n\Delta_n}}
            \sum_{i=1}^n \int_{(i-1)\Delta_n}^{i\Delta_n}
            \left[f_2^*(X_s) - f_2^*(X_{t^n_{i-1}})\right] \dd s
	=	 \sqrt{n\Delta_n} \left(\frac{1}{n}\sum_{i=1}^n f_2^*(X_{t^n_{i-1}})\right) + o_{\Pnull}(1), \nonumber
\end{eqnarray}
where we only use that $f_2^* \in \C^2_p(S)$. A proof that the second
term in \eqref{eqn:negligible} is asymptotically negligible under
$\Pnull$ is contained in the proof of Proposition~3.4 in
\cite{ped}. Furthermore, by Proposition~3.3 in the same paper,
$\Gnull\left(U_0(f_2^*)\right) = -f_2^*$ under Condition
\ref{cond:G.1}, and, therefore, by Itô's formula, 
\begin{eqnarray*}
\U_0(f_2^*)(X_t)
&=& \U_0(f_2^*)(X_0) + \int_0^t \Gnull (U_0(f_2^*))(X_s) \dd s + \int_0^t \partial_x U_0(f_2^*)(X_s)b(X_s;\theta_0) dB_s \\
&=& \U_0(f_2^*)(X_0) - \int_0^t f_2^*(X_s) \dd s + \int_0^t \partial_x U_0(f_2^*)(X_s)b(X_s;\theta_0) dB_s.
\end{eqnarray*}
As a consequence,
\begin{eqnarray*}
\sqrt{n\Delta_n}\left(\frac{1}{n} \sum_{i=1}^n f_2^*(X_{t^n_{i-1}})\right)
&=	& \frac{1}{\sqrt{n\Delta_n}} \int_0^{n\Delta_n} f_2^*(X_s) \dd s + o_{\Pnull}(1) \\
&=& \frac{1}{\sqrt{n\Delta_n}} \sum_{i=1}^n \int_{(i-1)\Delta_n}^{i\Delta_n} \partial_x U_0(f_2^*)(X_s)b(X_s;\theta_0) dB_s + o_{\Pnull}(1),
\end{eqnarray*}
and hence
\begin{eqnarray*}
\lefteqn{\frac{1}{\sqrt{n\Delta_n}} \sum_{i=2}^n
  g_2(\Delta_n,Y_i,Y_{i-1};\theta_0)} \\
&	=	&	\sqrt{n\Delta_n}\left(\frac{1}{n} \sum_{i=1}^n
            f_2^*(X_{t^n_{i-1}})\right) + \frac{1}{\sqrt{n\Delta_n}}
            \sum_{i=1}^n f(X_{t^n_{i-1}}) M_i(\theta_0) +
            o_{\Pnull}(1) \\ 
&	=	&	\frac{1}{\sqrt{n\Delta_n}} \sum_{i=1}^n
            \int_{(i-1)\Delta_n}^{i\Delta_n} \left[\partial_x
            U_0(f_2^*)(X_s)+f(X_{t^n_{i-1}})\partial_x f(X_s)\right]
            b(X_s;\theta_0) dB_s + o_{\Pnull}(1). 
\end{eqnarray*}
At this point, the asymptotic normality in \eqref{eq:g2} can be
shown by applying the central limit theorem for martingale difference
arrays, see e.g.\ \cite{hallheyde} or \cite{slt}; for details see pp.\
507-508 in \cite{ped}.  The joint normality in (\ref{eq:AN.H}) follows
by the Cramér-Wold device. 
\end{proof}

	\section{Concluding remarks and extensions}\label{sec:conclusion}

For integrated diffusions observed on $[0,1]$, \cite{id-2008} prove that the statistical model satisfies the LAMN property and that the optimal rate of convergence of estimators of a parameter in the diffusion coefficient is $1/\sqrt{n}$. The optimal rates for integrated diffusion models under the high-frequency/infinite horizon scenario considered in this paper are not known, but the minimum contrast estimators in \cite{id-2006} attain a rate of $1/\sqrt{n\Delta_n}$ for parameters in the drift and $1/\sqrt{n}$ for diffusion parameters under this scenario (similar to the rate optimal estimators for discretely observed diffusions in \cite{eed-2017}), so presumably these rates are optimal. However, as we do not distinguish between drift and diffusion parameters in this paper, the $1/\sqrt{n\Delta_n}$ rate of our parameters is all we could hope for.

An interesting extension would be to introduce a jump component in the dynamics of $(X_t)$. Such an extension has the particular feature that jumps in $(X_t)$ lead to changes in the trend of $(I_t)$ and \emph{not} to path discontinuities. As a consequence, threshold estimators developed for processes with jumps observed at high-frequency (see e.g. \cite{tej}) are not directly transferable. A general test for the presence of volatility jumps using change-point theory was proposed by \cite{cpa}. How and whether the same principle can be applied for parametric inference is an interesting topic for future research.

        \section*{Acknowledgement}

        Emil S. Jørgensen gratefully acknowledges financial support
        during a research stay from the Stevanovich Center for Financial
        Mathematics, University of Chicago.

\end{document}